\documentclass{paper}

\usepackage{amsmath,amsthm}
\usepackage{amssymb}
\usepackage{bm}
\usepackage{booktabs}
\usepackage{array}
\usepackage{multirow}
\usepackage{enumitem}
\usepackage[colorlinks,linkcolor=blue,citecolor=blue,urlcolor=blue]{hyperref}
\usepackage{cleveref}

\newtheorem{theorem}{Theorem}[section]
\newtheorem{lemma}[theorem]{Lemma}

\newtheorem{proposition}[theorem]{Proposition}
\theoremstyle{definition}
\newtheorem{definition}[theorem]{Definition}
\newtheorem{assumption}[theorem]{Assumption}
\theoremstyle{remark}
\newtheorem{remark}[theorem]{Remark}
\crefname{assumption}{Assumption}{Assumptions}
\Crefname{assumption}{Assumption}{Assumptions}

\newcommand{\headers}[2]{}
\newcommand{\email}[1]{\texttt{#1}}
\newenvironment{MSCcodes}{\small\paragraph*{MSC codes:}}{\par\bigskip}
\numberwithin{equation}{section}

\newcommand{\R}{\mathbb{R}}
\newcommand{\N}{\mathbb{N}}
\newcommand{\dx}{\,dx}
\newcommand{\norm}[1]{\|#1\|}
\newcommand{\snorm}[2]{\|#1\|_{#2}}
\newcommand{\VN}{V_N}
\newcommand{\PN}{P_N}
\newcommand{\PNo}{P_N^{0}}
\newcommand{\lbd}{\underline{\lambda}}
\newcommand{\Hone}{H^1}

\DeclareMathOperator{\spn}{span}

\headers{Lower Eigenvalue Bounds for Spectral Galerkin Methods}{X. Liu}

\title{Guaranteed Lower Eigenvalue Bounds for Spectral Galerkin Methods
  with Application to Schr\"odinger Operators \thanks{Submitted to the editors \today.}}

\author{Xuefeng Liu\thanks{Tokyo Woman's Christian University, 2-6-1
    Zempukuji, Suginami-ku, Tokyo 167-8585, Japan
    (\email{xfliu@lab.twcu.ac.jp}).}}

\begin{document}
\maketitle

\begin{abstract}
Spectral Galerkin methods are renowned for high-precision eigenvalue
approximation, yet a rigorous lower bound obtained \emph{directly} from a
spectral discretisation has remained unavailable: the classical Kato and
Weinstein--Temple enclosures do apply, but require a~priori information on
a neighbouring eigenvalue.  This paper resolves the issue by extending the author's
projection-based framework for guaranteed lower eigenvalue bounds---so far
realised only through finite element methods---to conforming spectral
Galerkin methods.  For trial spaces of exact
eigenfunctions the required projection constant is the closed-form optimal
value $C_N=\lambda_{M+1}^{-1/2}$, the inverse square root of the first
omitted eigenvalue.  For $-\Delta+V$ with $0\le V\in L^\infty$, a
\emph{projection-gap estimate} yields an explicit constant for the
standard Galerkin matrix (exact at $V=0$), and a composite discretisation
removes the $\norm{V}_{L^\infty}$-dependence for large potentials.  With
Neumann domain truncation these give certified two-sided bounds on
$\R^d$; for two benchmark potentials on $\R^2$ the spectral enclosures
match or surpass certified finite element ones at two orders of magnitude
fewer degrees of freedom.  The same auxiliary-projector mechanism extends
to singular potentials with an unbounded $L^\infty$ norm---in particular to
attractive Coulomb singularities in three dimensions, via a localised Hardy
inequality---which we develop in a companion paper.
\end{abstract}

\begin{keywords}
eigenvalue bounds, spectral Galerkin method, verified computation,
projection error constant, Schr\"odinger operator, domain truncation
\end{keywords}

\begin{MSCcodes}
65N25, 65N35, 65G20, 35P15, 81Q05
\end{MSCcodes}

\section{Introduction}
\label{sec:intro}

The eigenvalues of self-adjoint differential operators govern vibration
frequencies, stability thresholds, spectral gaps of quantum systems, and the
a~posteriori error constants used in computer-assisted proofs for partial
differential equations.  The Rayleigh--Ritz principle makes certified
\emph{upper} bounds easy: every conforming Galerkin eigenvalue lies above
the corresponding exact one.  Certified \emph{lower} bounds are
substantially harder.  Classical approaches---Temple--Kato bounds
\cite{Temple1928,Kato1949}, the Lehmann--Goerisch method
\cite{Lehmann1963,BehnkeGoerisch1994}, Weinberger's and Kuttler--Sigillito's
constructions \cite{Weinberger1956,KuttlerSigillito1965}, and homotopy-based
enclosures \cite{Plum1991,NakaoPlumWatanabe2019}---require a~priori spectral
information (e.g.\ a rough lower bound for the next eigenvalue) or
problem-specific auxiliary constructions (e.g.\ a base problem with a closed-form eigensystem).

A different route was opened by the early ideas of Birkhoff et al.\ \cite{Birkhoff_etal1966}, Kikuchi--Liu \cite{Kikuchi+Liu2007, liu-kikuchi-2010}, Liu--Oishi \cite{LiuOishi2013}, Kobayashi \cite{Kobayashi2015}, and Carstensen et al.\ \cite{CarstensenGallistl2014,CarstensenGedicke2014}, all seeking eigenvalue bounds without a~priori spectral information. This approach is
formalised in an abstract Hilbert-space framework by Liu
\cite{Liu2015,Liu2024book}: if the Ritz projection $P_h$ onto the trial
space satisfies an explicit projection-error estimate
$\norm{(I-P_h)v}_b\le C_h\norm{(I-P_h)v}_a$, then every Galerkin eigenvalue
$\lambda_{k,h}$ yields, \emph{unconditionally}, the bound
\begin{equation}
\label{eq:liu-bound-intro}
  \lambda_k \;\ge\; \frac{\lambda_{k,h}}{1+C_h^2\,\lambda_{k,h}},
  \qquad k=1,2,\ldots,\dim V_h.
\end{equation}
No a~priori information on the spectrum is needed.  In existing
realisations of \eqref{eq:liu-bound-intro} the constant $C_h$ arises from the error estimation of finite element methods (FEM) on meshes: it is obtained
either from \emph{global error estimation via the hypercircle method or interpolation error constants for conforming FEMs} or from \emph{local interpolation estimates for non-conforming FEMs}.  Examples include the
Crouzeix--Raviart and enriched Crouzeix--Raviart elements for the Laplacian
and Schr\"odinger operators \cite{Liu2015,Liu2024book,LiuJSIAM2026},
Crouzeix--Raviart pairs for the Stokes problem \cite{XieXieLiu2018},
trace-type constants for the Steklov problem \cite{YouXieLiu2019}, and the
Fujino--Morley element for the biharmonic operator \cite{Liu2024book}; see
also \cite{CarstensenGedicke2014,CarstensenGallistl2014} for closely related
guaranteed FEM lower bounds.  These methods converge at the algebraic rate
dictated by low-order interpolation, typically $O(h^2)$ in the eigenvalue.

\emph{Spectral Galerkin methods}, which use global trial functions such as
trigonometric or orthogonal polynomial bases, achieve much faster
convergence of the Galerkin eigenvalues for smooth problems
\cite{ShenTangWang2011,Boyd2001}.  Yet they have remained outside the
guaranteed-lower-bound framework: the mechanism producing $C_h$ from local
mesh geometry has no spectral counterpart, and to our knowledge no explicit
projection constant for a spectral Galerkin space has been published.  This
paper closes that gap.

\subsection*{Contributions}

The contributions are fourfold.

\begin{enumerate}[label=(\roman*),leftmargin=2.2em]
\item \emph{The spectral projection constant (\cref{lem:spectral-CN}).}
For a trial space spanned by the first $M$ exact eigenfunctions, the
optimal projection-error constant is $C_N=\lambda_{M+1}^{-1/2}$ (with $\lambda_{M+1}$ the first omitted eigenvalue), obtained by a one-line Parseval
argument, with no mesh or interpolation theory.  Elementary in itself, it
is the building block for the two mechanisms below, where the spectrum is
no longer known.

\item \emph{The projection-gap mechanism for $-\Delta+V$, $V\ge0$
(\cref{thm:gap-CN}).}
The conceptual core.  When the trigonometric trial space is not an
eigenbasis, the modal-tail argument of (i) fails for the Ritz projection
$\PN$ of the full form; we instead control the \emph{projection gap}
$\PNo u-\PN u$ between the Laplacian (modal) projection $\PNo$ and $\PN$
through an auxiliary discrete solution operator
\cite[Chap.~4]{Liu2024book}. 

\item \emph{Removing the $\norm{V}_{L^\infty}$-dependence
(\cref{thm:schrodinger-main}).}
After truncation, confining potentials make the factor in (ii) large by
design.  A composite discretisation---a bandlimited under-approximation
sampled through a \emph{slaved weighted projection}---restores the
sharp constant $\nu_*^{-1/2}$ with \emph{no} $\norm{V}_{L^\infty}$
dependence.  With Neumann truncation
\cite{LiuJSIAM2026,NakaoPlumWatanabe2019} this gives certified two-sided
bounds on $\R^d$.

\item \emph{Extension to singular potentials.}
The same mechanism extends, through a fixed shift and a localised Hardy
inequality, to potentials with an unbounded $L^\infty$ norm---including
attractive Coulomb singularities in three dimensions, where the form
factor stays finite and the certified bound converges at the optimal
rate.  This application to quantum-chemistry potentials is developed in a
companion paper; \cref{rem:coulomb-companion} previews the mechanism.
\end{enumerate}

\subsection*{Why spectral instead of FEM?}

For the problem class targeted here---smooth potentials on simple truncated
domains---the comparison is summarised as follows.  The FEM constant
$C_h=O(h)$ is typically produced by local interpolation estimates whose constants
must be derived element by element, and the certified gap closes at
$O(h^2)$, i.e.\ $O(\mathrm{DOF}^{-2/d})$ in dimension $d$.  The spectral
constant $C_N=O(1/N)$ is a single closed-form number, the
upper bounds converge spectrally (super-algebraically) for smooth
potentials, and the certified gap closes at $O(N^{-2})$ with $N$ the number
of modes \emph{per direction}, i.e.\ $O(\mathrm{DOF}^{-2/d})$ with a far
smaller constant and with exact assembly (no quadrature error) for
polynomial potentials.   In return, FEM retains the advantage on
complicated geometries and for strongly singular potentials; the two
approaches are complementary.

\subsection*{Relation to existing work}

The framework \eqref{eq:liu-bound-intro} originates in
\cite{LiuOishi2013,Liu2015}; the monograph \cite{Liu2024book} gives a
systematic treatment, including the auxiliary-operator technique for
positive zeroth-order coefficients that inspires the projection-gap
analysis of \cref{sec:schrodinger}.  Domain truncation with Neumann
lower bounds for confining potentials is taken from \cite{LiuJSIAM2026};
the treatment of Coulomb-type singular potentials---beyond the scope of this paper and taken up in the
companion paper---is from \cite{LiuNM2026}.
Standard
references for spectral methods are \cite{ShenTangWang2011,Boyd2001}.
To our knowledge, the spectral constant $C_N=\lambda_{M+1}^{-1/2}$, the
projection-gap constant of \cref{thm:gap-CN}, and the composite theorem
\cref{thm:schrodinger-main} are new.

\subsection*{Outline}

\Cref{sec:framework} states and proves the abstract lower-bound theorem.
\Cref{sec:spectral} treats eigenbasis trial spaces: the sharp constant
$C_N$, the closed-form one-dimensional bounds, and the tensor-product
construction on rectangles.  \Cref{sec:schrodinger} develops the
projection-gap mechanism for Schr\"odinger operators with $V\ge0$, and
then the composite refinement for large potentials together with the
truncation machinery for unbounded domains.  \Cref{sec:numerics} presents one- and two-dimensional
numerical results and the comparison with FEM.  \Cref{sec:sharpness}
proves sharpness, monotone convergence, and the index-saturation limit.  \Cref{sec:conclusion}
concludes.

\section{The abstract lower-bound theorem}
\label{sec:framework}

Throughout this section $\widehat V$ is a real vector space equipped with
two symmetric, positive semi-definite bilinear forms
$\widehat a(\cdot,\cdot)$ and $\widehat b(\cdot,\cdot)$; we write
$\snorm{\cdot}{\widehat a}$ and $\snorm{\cdot}{\widehat b}$ for the induced
seminorms.  The reader may keep in mind the conforming model case
$\widehat V=V$ a Hilbert space, $\widehat a=a$ an inner product and
$\widehat b(u,v)=b(\gamma u,\gamma v)$ with a compact operator
$\gamma:V\to W$ into a second Hilbert space, as in
\cite{Liu2015,Liu2024book}; the slightly more general seminorm setting
costs nothing in the proof and is needed for the composite-space argument
of \cref{subsec:composite}.

\begin{assumption}
\label{ass:abstract}{\quad}

\begin{enumerate}[label=(A\arabic*),ref=(A\arabic*),leftmargin=2.4em]
\item \label{a1} There exist countably many \emph{exact eigenpairs}
$(\lambda_j,u_j)\in(0,\infty)\times\widehat V$, $j=1,2,\ldots$, with
$0<\lambda_1\le\lambda_2\le\cdots$, satisfying the Gram relations
\begin{equation}
\label{eq:exact-evp}
  \widehat b(u_i,u_j)=\delta_{ij},
  \qquad
  \widehat a(u_i,u_j)=\lambda_j\,\delta_{ij},
  \qquad i,j\ge1 .
\end{equation}
(In the conforming model case these follow from the eigenvalue equation
$\widehat a(u_j,v)=\lambda_j\widehat b(u_j,v)$ for all $v$, after
$\widehat b$-normalisation; only \eqref{eq:exact-evp} is used below.)
\item \label{a2} $\widehat V_N\subset\widehat V$ is a finite-dimensional
subspace on which $\widehat a$ and $\widehat b$ are positive definite, and the
\emph{discrete eigenvalues} $0<\lambda_{1,N}\le\cdots\le\lambda_{M,N}$
($M=\dim\widehat V_N$) are defined by the matrix eigenvalue problem for
$(\widehat a,\widehat b)$ restricted to $\widehat V_N$.
\item \label{a3} For each $k\le M$ there is a linear map
$\Pi_N:E_k\to\widehat V_N$, $E_k:=\spn\{u_1,\ldots,u_k\}$, such that for
all $\phi\in E_k$,
\begin{align}
  &\widehat a(\Pi_N\phi,\,\phi-\Pi_N\phi) = 0
  \qquad\text{(Pythagoras property)},
  \label{eq:hyp-pythagoras}\\
  &\snorm{\phi-\Pi_N\phi}{\widehat b} \le
  C_N\,\snorm{\phi-\Pi_N\phi}{\widehat a}
  \qquad\text{(projection-error estimate)}.
  \label{eq:hyp-estimate}
\end{align}
\end{enumerate}
\end{assumption}

In the conforming model case, $\Pi_N$ is the Ritz ($\widehat
a$-orthogonal) projection and \eqref{eq:hyp-pythagoras} is automatic; the
formulation above additionally admits maps $\Pi_N$ that are orthogonal
projections only componentwise, which is exactly what the composite
construction of \cref{subsec:composite} provides.

\begin{theorem}[Projection-based lower bound; cf.\ {\cite[Thm.~2.1]{Liu2015}},
  {\cite[Chap.~3]{Liu2024book}}]
\label{thm:liu}
Under \cref{ass:abstract}, for each $k=1,\ldots,M$,
\begin{equation}
\label{eq:liu-bound}
  \lambda_k \;\ge\; \frac{\lambda_{k,N}}{1+C_N^2\,\lambda_{k,N}} .
\end{equation}
\end{theorem}

\begin{proof}
Write $\kappa_j=1/\lambda_j$, $\kappa_{j,N}=1/\lambda_{j,N}$ for the
reciprocal (exact and discrete) eigenvalues. In this notation,
the lower bound \eqref{eq:liu-bound} reads $\kappa_k\le \kappa_{k,N}+C_N^2$.
Define
$$
\widehat R(v):=\frac{\snorm{v}{\widehat b}^2}{\snorm{v}{\widehat a}^2} \text{ (whenever $\snorm{v}{\widehat a}>0$)}.
$$
 By \eqref{eq:exact-evp}, every
$\phi=\sum_{j\le k}c_ju_j\in E_k$
satisfies $\snorm{\phi}{\widehat b}^2=\sum c_j^2$ and
$\snorm{\phi}{\widehat a}^2=\sum\lambda_jc_j^2>0$, whence
\begin{equation}
\label{eq:exact-minmax}
  \min_{0\ne\phi\in E_k}\widehat R(\phi)=\kappa_k .
\end{equation}
By the max--min principle for the pencil
$(\widehat b,\widehat a)$ on $\widehat V_N$,
\begin{equation}
\label{eq:discrete-minmax}
  \kappa_{k,N}
  =\max_{\substack{S\subset\widehat V_N\\ \dim S=k}}\;
   \min_{0\ne v\in S}\widehat R(v) .
\end{equation}

\emph{Case 1: $\Pi_N$ is not injective on $E_k$.}  Pick $0\ne\phi\in E_k$
with $\Pi_N\phi=0$.  Then \eqref{eq:hyp-estimate} gives
$\snorm{\phi}{\widehat b}\le C_N\snorm{\phi}{\widehat a}$, so by
\eqref{eq:exact-minmax} $\kappa_k\le\widehat R(\phi)\le C_N^2\le
\kappa_{k,N}+C_N^2$.

\emph{Case 2: $\Pi_N$ is injective on $E_k$.}  Then
$S:=\Pi_N(E_k)\subset\widehat V_N$ has dimension $k$, and by
\eqref{eq:discrete-minmax} there exists $\phi^*\in E_k$, $\phi^*\ne0$,
with
\begin{equation}
\label{eq:phistar}
  \widehat R(\Pi_N\phi^*)\;\le\;\kappa_{k,N}.
\end{equation}
(Choose $\phi^*$ so that $\Pi_N\phi^*$ minimises $\widehat R$ over
$S\setminus\{0\}$; then $\min_{v\in S}\widehat R(v)\le\kappa_{k,N}$ by
\eqref{eq:discrete-minmax} applied to the candidate subspace $S$.)
Using the triangle inequality in $\snorm{\cdot}{\widehat b}$,
\eqref{eq:hyp-estimate}, \eqref{eq:phistar}, and the Cauchy--Schwarz
inequality in $\R^2$,
\[
  \snorm{\phi^*}{\widehat b}
  \le \snorm{\Pi_N\phi^*}{\widehat b}
     +\snorm{\phi^*-\Pi_N\phi^*}{\widehat b}
  \le \sqrt{\kappa_{k,N}}\,\snorm{\Pi_N\phi^*}{\widehat a}
     +C_N\,\snorm{\phi^*-\Pi_N\phi^*}{\widehat a}
\]
\[
  \le \sqrt{\kappa_{k,N}+C_N^2}\;
      \sqrt{\snorm{\Pi_N\phi^*}{\widehat a}^2
            +\snorm{\phi^*-\Pi_N\phi^*}{\widehat a}^2}
  \;=\;\sqrt{\kappa_{k,N}+C_N^2}\;\snorm{\phi^*}{\widehat a},
\]
where the last equality is the Pythagoras property
\eqref{eq:hyp-pythagoras}.  Hence
$\kappa_k\le\widehat R(\phi^*)\le\kappa_{k,N}+C_N^2$ by
\eqref{eq:exact-minmax}.

In both cases $\kappa_k\le\kappa_{k,N}+C_N^2$, which is
\eqref{eq:liu-bound} after inversion.
\end{proof}

\begin{remark}[Quality of the bound]
\label{rem:gap-formula}
The relative gap of \eqref{eq:liu-bound} is
\begin{equation}
\label{eq:relative-gap}
  \frac{\lambda_k-\lbd_{k,N}}{\lambda_k}
  \;\le\;
  \frac{C_N^2\lambda_{k,N}}{1+C_N^2\lambda_{k,N}},
  \qquad
  \lbd_{k,N}:=\frac{\lambda_{k,N}}{1+C_N^2\lambda_{k,N}} .
\end{equation}
The dimensionless product $C_N\sqrt{\lambda_{k,N}}$ governs everything:
the bound is useful precisely when $C_N^2\lambda_{k,N}\ll1$.  All the work
in applications goes into making $C_N$ explicit and small; this
is what \cref{sec:spectral,sec:schrodinger} accomplish for spectral
Galerkin methods.
\end{remark}

\begin{remark}[No a~priori information]
Unlike Temple--Kato or Lehmann--Goerisch bounds, \cref{thm:liu} needs no
lower bound for $\lambda_{k+1}$, no separation assumption, and no rough
localisation of the spectrum.  This unconditional character is inherited
by all bounds in this paper.
\end{remark}

\section{Spectral Galerkin spaces spanned by exact eigenfunctions}
\label{sec:spectral}

This section treats the conforming model case: $V$ a Hilbert space with
inner product $a$, $W$ a Hilbert space with inner product $b$,
$\gamma:V\to W$ compact with dense range, and the eigenproblem
\begin{equation}
\label{eq:conforming-evp}
  a(u,v)=\lambda\,b(\gamma u,\gamma v)\qquad\forall v\in V,
\end{equation}
with eigenvalues $0<\lambda_1\le\lambda_2\le\cdots\to\infty$ and
$a$-orthonormal eigenfunctions $\{\phi_j\}_{j\ge1}$ forming a complete
system in $V$ (we assume $b(\gamma\cdot,\gamma\cdot)$ positive definite on
$V$; the semi-definite case is treated in \cite[Chap.~3]{Liu2024book}).
With $\widehat V=V$, $\widehat a=a$,
$\widehat b=b(\gamma\cdot,\gamma\cdot)$, \cref{ass:abstract}\ref{a1} holds
after $b$-renormalisation of the $\phi_j$.

\subsection{The sharp spectral constant}

Let $\Lambda_N\subset\N$ be a finite \emph{retained index set} with
$|\Lambda_N|=M$ and define the eigenbasis trial space
\begin{equation}
\label{eq:VN-eigenbasis}
  \VN=\spn\{\phi_j:\,j\in\Lambda_N\}\subset V .
\end{equation}
The Ritz projection $\PN:V\to\VN$, $a(v-\PN v,w_N)=0$ for all
$w_N\in\VN$, is the modal truncation
$\PN v=\sum_{j\in\Lambda_N}a(v,\phi_j)\,\phi_j$.

\begin{lemma}[Spectral constant via modal tail estimate]
\label{lem:spectral-CN}
Let
\begin{equation}
\label{eq:lambda-star}
  \lambda_*:=\min\{\lambda_j:\,j\notin\Lambda_N\} .
\end{equation}
Then for all $v\in V$,
\begin{equation}
\label{eq:tail-estimate}
  \snorm{\gamma(I-\PN)v}{b}^2\;\le\;\frac1{\lambda_*}\,
  \snorm{(I-\PN)v}{a}^2 ,
\end{equation}
and the constant $C_N=\lambda_*^{-1/2}$ is optimal: it is the smallest
constant for which \eqref{eq:tail-estimate} holds.
\end{lemma}

\begin{proof}
By completeness, $(I-\PN)v=\sum_{j\notin\Lambda_N}a(v,\phi_j)\phi_j$, and
by \eqref{eq:conforming-evp}, $b(\gamma\phi_i,\gamma\phi_j)
=\delta_{ij}/\lambda_j$.  Parseval in both norms gives
\[
  \snorm{\gamma(I-\PN)v}{b}^2
  =\sum_{j\notin\Lambda_N}\frac{|a(v,\phi_j)|^2}{\lambda_j}
  \;\le\;\frac1{\lambda_*}\sum_{j\notin\Lambda_N}|a(v,\phi_j)|^2
  =\frac1{\lambda_*}\snorm{(I-\PN)v}{a}^2 .
\]
For optimality take $v=\phi_{j_*}$ with $\lambda_{j_*}=\lambda_*$: then
$\PN v=0$ and equality holds in \eqref{eq:tail-estimate}.
\end{proof}

Specialising the abstract bound \eqref{eq:liu-bound} to the eigenbasis
space, where the Galerkin eigenvalues reproduce the exact ones
($\lambda_{k,N}=\lambda_{j(k)}$, $j(k)$ the $k$-th smallest index in
$\Lambda_N$), gives
\begin{equation}
\label{eq:eigenbasis-bound}
  \lambda_k\;\ge\;\lbd_{k,N}
  =\frac{\lambda_{k,N}\,\lambda_*}{\lambda_{k,N}+\lambda_*},
  \qquad k=1,\ldots,M .
\end{equation}
Since $\lambda_{k,N}=\lambda_k$ is already exact here,
\eqref{eq:eigenbasis-bound} is not a means of bounding unknown
eigenvalues; it serves only to expose, on a problem with a fully known
spectrum, how much the bound formula concedes relative to the exact
value---a calibration of the constant $C_N$ that we use as the benchmark
below.

Both the FEM constant $C_h$ and the spectral constant $C_N$ feed the same
formula \eqref{eq:liu-bound}, but their origins differ fundamentally:
for non-conforming FEMs, $C_h$ is \emph{local} (element-by-element interpolation; e.g.\
$C_h\le0.1893\,h$ for the Crouzeix--Raviart element in 2D
\cite{Liu2015}) while $C_N$ is \emph{global}, determined by the spectral
gap at the truncation boundary, and is \emph{exact} for eigenbasis
spaces.

\subsection{A fundamental benchmark: the Dirichlet Laplacian}
\label{subsec:laplacian-benchmark}
For $-u''=\lambda u$ on $(0,1)$, $u(0)=u(1)=0$:
$V=H_0^1(0,1)$, $a(u,v)=\int_0^1u'v'\dx$, $W=L^2(0,1)$, $\gamma$ the
embedding.  Exact eigenpairs $\lambda_k=k^2\pi^2$,
$\phi_k\propto\sin(k\pi x)$.  With
$\VN=\spn\{\sin(j\pi x)\}_{j=1}^N$ (so $\Lambda_N=\{1,\ldots,N\}$,
$\lambda_*=(N+1)^2\pi^2$), \eqref{eq:eigenbasis-bound} gives the
closed-form certified bound
\begin{equation}
\label{eq:1d-closed-form}
  \lambda_k\;\ge\;\lbd_{k,N}=\frac{k^2\pi^2}{1+k^2/(N+1)^2},
  \qquad k=1,\ldots,N,
\end{equation}
with relative gap exactly $k^2/(k^2+(N+1)^2)=O(k^2/N^2)$---a closed-form
expression involving no numerical eigenvalue computation, hence verifiable
by inspection.

The Laplacian case is exactly solvable and serves two purposes: it
calibrates the framework (the bounds \eqref{eq:1d-closed-form} are sharp;
see \cref{sec:sharpness}), and it supplies the comparison eigenvalues
$\nu_j$ used in the next section, where the actual target operator has a
potential and is no longer exactly solvable.

\section{Schr\"odinger operators with nonnegative potentials}
\label{sec:schrodinger}

We now turn to the main subject of the paper, the Schr\"odinger eigenvalue
problem. Given a box domain $\Omega:=\prod_{i=1}^d(-L_i,L_i)\subset\R^d$
and $\mathcal V=\Hone(\Omega)$ (Neumann) or $\mathcal V=H_0^1(\Omega)$
(Dirichlet), find $u \in \mathcal V$ and $\mu \in \R$ such that
\begin{equation}
\label{eq:schro-evp}
  a(u,v)=\mu\,(u,v)_{L^2(\Omega)}
  \qquad\forall v\in\mathcal V,
\end{equation}
where $a(u,v):=a_0(u,v)+(Vu,v)_{L^2}$ and $a_0(u,v):=(\nabla u,\nabla v)$.
Throughout this section, it is assumed that
\begin{equation}
\label{eq:V-assumption}
  V\in L^\infty(\Omega),\qquad V\ge0\ \text{a.e.},\qquad
  V>0\ \text{on some open subset of }\Omega .
\end{equation}
The eigenvalues of \eqref{eq:schro-evp} are denoted
$0<\mu_1\le\mu_2\le\cdots$.  The Neumann setting is the one required by
domain truncation for operators on $\R^d$ (\cref{subsec:truncation}).

It will be convenient to quantify the \emph{fluctuation} of the potential
over the box by its variance under the uniform measure on $\Omega$.  Writing
$\langle f\rangle_\Omega:=|\Omega|^{-1}\int_\Omega f\,dx$, set
\begin{equation}
\label{eq:var-def}
  \operatorname{Var}_\Omega(V)
  :=\langle V^2\rangle_\Omega-\langle V\rangle_\Omega^2
  =\frac1{|\Omega|}\int_\Omega\bigl(V-\langle V\rangle_\Omega\bigr)^2\,dx
  \;\ge\;0,
\end{equation}
the mean-square deviation of $V$ from its own box-average.  The fluctuation $\operatorname{Var}_\Omega(V)$ vanishes for a constant potential and is insensitive to the additive
constant in $V$, isolating the part of $V$ that couples distinct Laplacian
modes.  This is the quantity that will govern the size of the refined form
factor $\eta_V'$ (\cref{rem:opt-etaV}).

\medskip

Let $\{\varphi_j\}_{j}$ denote the $L^2$-orthonormal eigenfunctions of the
Laplacian on $\Omega$ with the chosen boundary condition (tensor cosines
for Neumann, tensor sines for Dirichlet) and $\nu_j$ the corresponding
eigenvalues; on a box these are explicit.  For Neumann conditions,
\[
  \varphi_{\bm m}(x)=\prod_{i=1}^d c_{m_i}
  \cos\Bigl(\frac{m_i\pi(x_i+L_i)}{2L_i}\Bigr),
  \qquad
  \nu_{\bm m}=\sum_{i=1}^d\frac{m_i^2\pi^2}{4L_i^2},
  \qquad \bm m\in\N_0^d,
\]
with $c_0=(2L_i)^{-1/2}$, $c_m=L_i^{-1/2}$ ($m\ge1$); for Dirichlet,
sines with $\bm m\in\N^d$ and $\nu_{\bm m}=\sum m_i^2\pi^2/(4L_i^2)$.
Fix the retained index set $\Lambda_N=\{\bm m:\,m_i\le N\}$ and set
\begin{equation}
\label{eq:VN-cos}
  \VN:=\spn\{\varphi_{\bm m}:\bm m\in\Lambda_N\},
  \qquad M=\dim\VN,
  \qquad
  \nu_*:=\min_{\bm m\notin\Lambda_N}\nu_{\bm m} .
\end{equation}
The \emph{standard spectral Galerkin discretisation} of
\eqref{eq:schro-evp} seeks $(\lambda_{k,N},u_N)\in\R\times\VN$ with
\begin{equation}
\label{eq:schro-galerkin}
  a(u_N,v_N)=\lambda_{k,N}\,(u_N,v_N)\qquad\forall v_N\in\VN ,
\end{equation}
i.e.\ the dense symmetric matrix eigenproblem $(K+P)\mathbf c
=\lambda_{k,N}\mathbf c$ with diagonal kinetic part
$K_{\bm m\bm m}=\nu_{\bm m}$ and full potential matrix
$P_{\bm m\bm m'}=(V\varphi_{\bm m},\varphi_{\bm m'})$.  By conformity,
$\lambda_{k,N}\ge\mu_k$: the Galerkin eigenvalues are certified
\emph{upper} bounds, converging spectrally fast for smooth $V$
\cite{ShenTangWang2011}.

\subsection{Two projections onto \texorpdfstring{$\VN$}{VN} and the
projection gap}
\label{subsec:gap}

Two projections onto the same trial space play distinct roles:
\begin{itemize}[leftmargin=2em]
\item [a)] the \emph{Laplacian (modal) projection}
$\PNo:\mathcal V\to\VN$, i.e.\ the $L^2$-orthogonal truncation of the
$\{\varphi_{\bm m}\}$-expansion; it is simultaneously orthogonal for
$a_0$ (gradients of distinct modes are orthogonal), it is the best
approximation in the $a_0$-seminorm, and it obeys the exact spectral-gap
estimate
\begin{equation}
\label{eq:modal-tail}
  \norm{u-\PNo u}_{L^2}\;\le\;\nu_*^{-1/2}\,\norm{\nabla(u-\PNo u)}_{L^2},
\end{equation}
since $u-\PNo u$ contains only modes with $\nu_{\bm m}\ge\nu_*$;
\item [b)] the \emph{Ritz projection} $\PN:\mathcal V\to\VN$ of the full form,
$a(u-\PN u,v_N)=0$ for all $v_N\in\VN$, which is the projection required
by \cref{thm:liu} but for which no modal tail estimate is available,
because the $\varphi_{\bm m}$ are not eigenfunctions of $-\Delta+V$.
\end{itemize}
When $V\ne0$ the two projections differ, and the whole difficulty is to
control the \emph{projection gap} $e:=\PNo u-\PN u\in\VN$.  The key is
the following algebraic identity, which transfers the gap to data
controlled by the \emph{modal} projection; it is the spectral counterpart
of the auxiliary-operator technique used in the FEM analysis of positive
zeroth-order coefficients in \cite[Chap.~4]{Liu2024book}.

\begin{lemma}[Gap identity]
\label{lem:gap-identity}
For every $u\in\mathcal V$ and every $v_N\in\VN$,
\begin{equation}
\label{eq:gap-identity}
  a(e,v_N)=\bigl(V(\PNo u-u),\,v_N\bigr)_{L^2},
  \qquad e=\PNo u-\PN u .
\end{equation}
\end{lemma}

\begin{proof}
$a(e,v_N)=a(\PNo u-u,v_N)+a(u-\PN u,v_N)$.  The second term vanishes by
the definition of $\PN$.  In the first term,
$a_0(\PNo u-u,v_N)=0$ by the $a_0$-orthogonality of the modal
projection, leaving only the potential part \eqref{eq:gap-identity}.
\end{proof}

The gap $e$ is therefore controlled by a single scalar $\eta_V$, the
\emph{form factor} of the potential relative to the trial space.

\begin{definition}[Potential form factor]
\label{def:etaV}
For the form $a$ on $\mathcal V$, let $\eta_V\ge0$ denote any constant
such that
\begin{equation}
\label{eq:etaV}
  |(Vw,v_N)|\;\le\;\eta_V\,\norm{w}_{L^2}\,\norm{v_N}_a
  \qquad\forall\,w\in L^2(\Omega),\ v_N\in\VN .
\end{equation}
\end{definition}

\begin{lemma}[Projection gap]
\label{lem:proj-gap}
Let $\lambda_{1,N}>0$ be the smallest Galerkin eigenvalue of
\eqref{eq:schro-galerkin}.  If \eqref{eq:etaV} holds, then for all
$u\in\mathcal V$,
\begin{equation}
\label{eq:proj-gap-est}
  \norm{\PNo u-\PN u}_{L^2}
  \;\le\;
  g\;\nu_*^{-1/2}\,\norm{\nabla(u-\PNo u)}_{L^2} ,
  \qquad
  g:=\lambda_{1,N}^{-1/2}\,\eta_V .
\end{equation}
\end{lemma}

\begin{proof}
Testing the gap identity \eqref{eq:gap-identity} with
$v_N=e:=\PNo u-\PN u\in\VN$ and using \eqref{eq:etaV} with
$w=\PNo u-u$,
\[
  \norm{e}_a^2=(V(\PNo u-u),e)
  \le\eta_V\,\norm{u-\PNo u}_{L^2}\,\norm{e}_a ,
\]
hence $\norm{e}_a\le\eta_V\norm{u-\PNo u}_{L^2}$.  The discrete
Rayleigh bound $\norm{e}_{L^2}\le\lambda_{1,N}^{-1/2}\norm{e}_a$ and the
modal-tail estimate \eqref{eq:modal-tail} then give
\eqref{eq:proj-gap-est}.
\end{proof}

\begin{lemma}[Form factor for bounded potentials]
\label{lem:bounded-etaV}
Let $0\le V\in L^\infty(\Omega)$ and set
\begin{equation}
\label{eq:X-def}
  X:=\frac{\norm{V}_{L^\infty}}{\lambda_{1,N}} .
\end{equation}
Then the form-factor inequality \eqref{eq:etaV} holds with each of the two
constants
\begin{equation}
\label{eq:bounded-etaV-two}
  \eta_V^{(1)}=\norm{V}_{L^\infty}\,\lambda_{1,N}^{-1/2},
  \qquad
  \eta_V^{(2)}=\norm{V}_{L^\infty}^{1/2}.
\end{equation}
Consequently the associated gap constant $g=\lambda_{1,N}^{-1/2}\eta_V$ of
\cref{lem:proj-gap} may be taken as
\begin{equation}
\label{eq:g-min}
  g=\min\!\bigl(X,\sqrt X\bigr),
  \qquad\text{so}\qquad
  g^2=\min\!\bigl(X^2,X\bigr),
\end{equation}
which is the value entering the projection-error constant $C_N^{(V)}$ of
\cref{thm:gap-CN}.
\end{lemma}

\begin{proof}
Both bounds estimate the bilinear form $(Vw,v_N)$ for $w\in L^2(\Omega)$
and $v_N\in\VN$; they differ in how the factor $V$ is distributed between
the two arguments.

\emph{First bound (mass on $w$, Rayleigh quotient on $v_N$).}
Since $0\le V\le\norm{V}_{L^\infty}$, 
\[
  |(Vw,v_N)|
  \;\le\;\norm{V}_{L^\infty}\,\norm{w}_{L^2}\,\norm{v_N}_{L^2}.
\]
Because $v_N\in\VN$ and $\lambda_{1,N}>0$ is the smallest Galerkin
eigenvalue, the discrete Rayleigh bound
$\norm{v_N}_{L^2}\le\lambda_{1,N}^{-1/2}\norm{v_N}_a$ holds; substituting it
gives
\[
  |(Vw,v_N)|
  \;\le\;\norm{V}_{L^\infty}\lambda_{1,N}^{-1/2}\,
  \norm{w}_{L^2}\,\norm{v_N}_a ,
\]
which is \eqref{eq:etaV} with $\eta_V^{(1)}=\norm{V}_{L^\infty}
\lambda_{1,N}^{-1/2}$.  The resulting gap constant is
$g=\lambda_{1,N}^{-1/2}\eta_V^{(1)}=\norm{V}_{L^\infty}/\lambda_{1,N}=X$.

\emph{Second bound (split $V^{1/2}$ symmetrically).}
Introduce the seminorm
$\norm{u}_V^2:=(Vu,u)$.  Since $V\ge0$, Cauchy--Schwarz for the
semi-inner product $(V\cdot,\cdot)$ gives
$|(Vw,v_N)|\le\norm{w}_V\,\norm{v_N}_V$.  The two seminorms are controlled
separately: on the one hand
$\norm{w}_V^2=(Vw,w)\le\norm{V}_{L^\infty}\norm{w}_{L^2}^2$, and on the
other hand
\[
  \norm{v_N}_V^2=(Vv_N,v_N)\;\le\;
  \norm{\nabla v_N}_{L^2}^2+(Vv_N,v_N)=a(v_N,v_N)=\norm{v_N}_a^2 ,
\]
the inequality using only $\norm{\nabla v_N}_{L^2}^2\ge0$.  Hence
\[
  |(Vw,v_N)|\;\le\;\norm{V}_{L^\infty}^{1/2}\,
  \norm{w}_{L^2}\,\norm{v_N}_a ,
\]
which is \eqref{eq:etaV} with $\eta_V^{(2)}=\norm{V}_{L^\infty}^{1/2}$,
independent of $\lambda_{1,N}$.  The resulting gap constant is
$g=\lambda_{1,N}^{-1/2}\eta_V^{(2)}
=\bigl(\norm{V}_{L^\infty}/\lambda_{1,N}\bigr)^{1/2}=\sqrt X$.

\emph{Combining.}  Both constants are admissible in \eqref{eq:etaV}, so
the smaller resulting $g$ may be used, giving
$g=\min(X,\sqrt X)$.  The two regimes exchange dominance at $X=1$: for
$X\le1$ (the potential is moderate relative to $\lambda_{1,N}$) one has
$X\le\sqrt X$, so the first bound wins and $g=X$; for $X\ge1$ (a large
potential) $\sqrt X\le X$, so the second bound wins and $g=\sqrt X$.  In
both cases $g^2=\min(X^2,X)$, which is precisely the term appearing under
the square root of $C_N^{(V)}=\sqrt{(1+\min(X^2,X))/\nu_*}$ in
\cref{thm:gap-CN}.
\end{proof}

\begin{lemma}[Tail-refined form factor]
\label{lem:opt-etaV}
Define the \emph{tail-refined form factor}
\begin{equation}
\label{eq:etaV-tail}
  \eta_V':=\sup_{0\ne v_N\in\VN}
   \frac{\norm{(I-\PNo)Vv_N}_{L^2}}{\norm{v_N}_a},
\end{equation}
the largest relative $L^2$-size of the part of $Vv_N$ that is not resolved
by $\VN$.  Then the projection-gap estimate \eqref{eq:proj-gap-est} of
\cref{lem:proj-gap}---and consequently the projection-error constant
$C_N^{(V)}$ of \cref{thm:gap-CN}---hold with $g$ replaced by
\begin{equation}
\label{eq:g-tail}
  g'=\lambda_{1,N}^{-1/2}\,\eta_V',
  \qquad\text{i.e. } (g')^2=\eta_V'^2/\lambda_{1,N}.
\end{equation}
\end{lemma}

\begin{proof}
In the proof of \cref{lem:proj-gap} the form factor is used only to bound
$(Vw,v_N)$ for the modal residual $w=\PNo u-u\in\VN^{\perp}$.  For such $w$,
$L^2$-orthogonality to $\VN$ removes the resolved part of $Vv_N$, so
$(Vw,v_N)=(w,Vv_N)=\bigl(w,(I-\PNo)Vv_N\bigr)$, and by Cauchy--Schwarz
together with the definition \eqref{eq:etaV-tail},
\[
  |(Vw,v_N)|\le\norm{w}_{L^2}\,\norm{(I-\PNo)Vv_N}_{L^2}
  \le\eta_V'\,\norm{w}_{L^2}\,\norm{v_N}_a .
\]
This is exactly the inequality \eqref{eq:etaV} restricted to the residual
$w\in\VN^{\perp}$, which is all that the proof of \cref{lem:proj-gap}
invokes.  Its conclusion \eqref{eq:proj-gap-est}, and hence
\cref{thm:gap-CN}, therefore hold with $g$ replaced by
$g'=\lambda_{1,N}^{-1/2}\eta_V'$.
\end{proof}

\begin{remark}[Computation of $\eta_V'$]
\label{rem:etaV-compute}
The supremum \eqref{eq:etaV-tail} is a standard generalised eigenvalue
problem on the same matrices used to assemble the operator.  Let
$\{\varphi_{\bm i}\}$ be the $L^2$-orthonormal basis of $\VN$, form the
standard Galerkin matrix $A=K+P$ with $P_{\bm i\bm j}=(V\varphi_{\bm i},
\varphi_{\bm j})$, and let $Q_{\bm i\bm j}=(V^2\varphi_{\bm i},
\varphi_{\bm j})$ be the Galerkin matrix of $V^2$ (assembled from the same
closed-form moments, exactly for polynomial $V$).  Writing
$v_N=\sum_{\bm i}c_{\bm i}\varphi_{\bm i}$, the $L^2$-orthonormality gives
$\norm{v_N}_a^2=c^{\!\top}\!Ac$, $\norm{Vv_N}_{L^2}^2=c^{\!\top}\!Qc$, and
$\norm{\PNo Vv_N}_{L^2}^2=c^{\!\top}\!P^2c$, so
\[
  \norm{(I-\PNo)Vv_N}_{L^2}^2
  =\norm{Vv_N}_{L^2}^2-\norm{\PNo Vv_N}_{L^2}^2
  =c^{\!\top}(Q-P^2)c .
\]
Here $Q-P^2$ is the Gram matrix on $\VN$ of the residuals $(I-\PNo)Vv_N$,
hence symmetric positive semidefinite.  Maximising the resulting Rayleigh
quotient over $c\ne0$ turns \eqref{eq:etaV-tail} into the closed form
\begin{equation}
\label{eq:etaV-tail-matrix}
  \eta_V'=\sqrt{\lambda_{\max}(Q-P^2,\,A)},
\end{equation}
the square root of the largest eigenvalue of the symmetric--definite pencil
$(Q-P^2,\,A)$.  Feeding $\eta_V'$ into \cref{thm:gap-CN} thus costs one
extra largest-eigenvalue solve on the same standard matrix $K+P$.
\end{remark}

\begin{remark}[Efficiency of $\eta_V'$ under refinement]
\label{rem:opt-etaV}
Because $\eta_V'$ measures only the unresolved part of $Vv_N$, it is set by
the potential's fluctuation over the box relative to the spectral gap,
$\eta_V'^2\sim\operatorname{Var}_\Omega(V)/\nu_*$.  The fluctuation is fixed
while $\nu_*$ grows with $N$, so unlike the $N$-independent
$\min(X,\sqrt X)$ of \cref{lem:bounded-etaV} the constant $\eta_V'$
\emph{shrinks under refinement}.  This is borne out by
\cref{tab:formfactor}: the tail-refined $(g')^2=\eta_V'^2/\lambda_{1,N}$
falls from $16.4$ to $12.0$ for $V_1$ and from $326$ to $290$ for $V_2$ as
$N$ grows from $48$ to $64$, so its lower bound improves faster than the
crude $\min(X,\sqrt X)$ and closes on the composite bound;
\cref{subsec:gap-limitation,subsec:2d-benchmarks} give the details.
\end{remark}

\begin{remark}[Relation to the auxiliary-operator formulation]
\label{rem:GN}
The linear factor $\eta_V=\norm{V}_{L^\infty}\lambda_{1,N}^{-1/2}$ can
equivalently be obtained through the discrete solution operator
$G_N:\VN\to\VN$, $a(G_N\phi_N,v_N)=(\phi_N,v_N)$, with
$\norm{G_N\phi_N}_a\le\lambda_{1,N}^{-1/2}\norm{\phi_N}_{L^2}$; this is
the spectral form of the FEM technique for positive zeroth-order
coefficients in \cite[Chap.~4]{Liu2024book}.  For Dirichlet conditions
an \emph{a~priori} variant replaces $\lambda_{1,N}$ by the first
Laplacian eigenvalue $\nu_1>0$, giving the computation-free factor
$\eta_V=\norm{V}_{L^\infty}\nu_1^{-1/2}$; since $\lambda_{1,N}\ge\nu_1$
the computable factor is at least as sharp, and it is the only option for
Neumann conditions, where $\nu_1=0$.  The abstract form
\eqref{eq:etaV} is what makes the singular case accessible: $\eta_V$ can
stay finite even when $\norm{V}_{L^\infty}=\infty$, as for the Coulomb
potentials previewed in \cref{rem:coulomb-companion}.
\end{remark}

\begin{theorem}[Projection-error constant for $-\Delta+V$, standard
matrix]
\label{thm:gap-CN}
Let \eqref{eq:V-assumption} hold.  Then for all $u\in\mathcal V$,
\begin{equation}
\label{eq:gap-CN}
  \norm{(I-\PN)u}_{L^2}\;\le\;C_N^{(V)}\,\norm{(I-\PN)u}_a,
  \qquad
  C_N^{(V)}:=\sqrt{\frac{1+\min(X^2,X)}{\nu_*}},
\end{equation}
with $X$ as in \eqref{eq:X-def}.  Consequently, with the eigenvalues
$\lambda_{k,N}$ of the standard Galerkin matrix $K+P$,
\begin{equation}
\label{eq:gap-lower-bound}
  \mu_k\;\ge\;\frac{\lambda_{k,N}}{1+(C_N^{(V)})^2\,\lambda_{k,N}},
  \qquad k=1,\ldots,M .
\end{equation}
\end{theorem}

\begin{proof}
Write $w:=(I-\PN)u$ and decompose $w=(u-\PNo u)+e$ with
$e=\PNo u-\PN u\in\VN$.  Since $u-\PNo u\perp_{L^2}\VN$, the two parts
are $L^2$-orthogonal, so by \eqref{eq:modal-tail} and
\cref{lem:proj-gap},
\[
  \norm{w}_{L^2}^2
  =\norm{u-\PNo u}_{L^2}^2+\norm{e}_{L^2}^2
  \;\le\;\frac{1+g^2}{\nu_*}\,\norm{\nabla(u-\PNo u)}^2 .
\]
Finally, $\PNo$ is the best approximation in the $a_0$-seminorm and
$V\ge0$, hence
\[
  \norm{\nabla(u-\PNo u)}\le\norm{\nabla(u-\PN u)}\le\norm{w}_a ,
\]
which gives \eqref{eq:gap-CN} with $g^2=\min(X^2,X)$.  The eigenvalue
bound follows from \cref{thm:liu} in the conforming setting
($\Pi_N=\PN$, the Ritz projection).
\end{proof}

\begin{remark}[Exact recovery at $V=0$ and the price of the potential]
\label{rem:V0-recovery}
For $V=0$ one has $X=0$ and $C_N^{(0)}=\nu_*^{-1/2}$, the \emph{optimal}
constant of \cref{lem:spectral-CN}: the gap mechanism degrades gracefully
to the sharp eigenbasis result.  For $V\ne0$ the constant exceeds the
optimum only through the factor $\sqrt{1+\min(X^2,X)}$; for moderate
potentials ($X\lesssim1$) this is a $O(X^2)$ perturbation.  Everything is
available from one assembly and one eigensolve of the \emph{standard}
matrix $K+P$: no second discretisation, no auxiliary constants, and the
factor improves automatically as $\lambda_{1,N}$ is computed.
\end{remark}

\subsection{Limitation: truncated confining potentials}
\label{subsec:gap-limitation}

After domain truncation (\cref{subsec:truncation}) the potential must
dominate the target eigenvalues outside $\Omega$, so
$\norm{V}_{L^\infty(\Omega)}\ge\sigma(\Omega)>\mu_k$ and typically
$X\gg1$.  Then $g=\sqrt X$ and
$(C_N^{(V)})^2=(1+X)/\nu_*$: the certified gap inflates by the factor
$1+X$.  For the benchmark $V_1=(|x|^2-1)^2$ on $(-4,4)^2$,
$\norm{V_1}_{L^\infty}=961$ and $X\approx538$, so at $N=48$
\cref{thm:gap-CN} with the crude factor $g^2=\min(X^2,X)=538$ certifies
only $\mu_1\ge0.496$ against the reference $1.78506$ ($72\%$ gap).  This
crude factor is far from optimal, however: the computable tail-refined
form factor $\eta_V'$ of \cref{lem:opt-etaV}---which sees only the small
leakage of $V_1v_N$ out of the resolved space, not
$\norm{V_1}_{L^\infty}$---lowers $g^2$ from $538$ to $16.4$ and raises
the \emph{same} standard-matrix bound to $\mu_1\ge1.647$ ($7.7\%$ gap),
within a modest factor of the composite value obtained below
(\cref{tab:formfactor}).
The composite discretisation of the next subsection achieves the unmodified constant $\nu_*^{-1/2}$ rigorously.

\subsection{The composite discretisation: removing
\texorpdfstring{$\norm{V}_{L^\infty}$}{the V-infinity factor}}
\label{subsec:composite}

The construction transplants the architecture of the composite enriched
Crouzeix--Raviart (CECR) method (see \cite[\S4.1]{Liu2024book} and
\cite{LiuJSIAM2026}) to the spectral setting.  In CECR, the potential is
replaced by a piecewise-constant under-approximation and is sampled
through the cell-average projection, whose interplay with the
(average-preserving) ECR interpolation produces the required
orthogonalities.  Here the corresponding roles are played by a
\emph{bandlimited under-approximation} of $V$ and a \emph{weighted
spectral projection of reduced degree}; the average-preservation property
of the ECR interpolant is replaced by an exact aliasing identity
(\cref{lem:slaving}).  We work with Neumann conditions (the case needed
in \cref{sec:numerics}); the Dirichlet version is analogous
(\cref{rem:dirichlet-version}).

\smallskip
\textbf{Bandlimited under-approximation.}
For $J\in\N_0$ let $T_J:=\spn\{\varphi_{\bm j}:\,0\le j_i\le J\}$ denote
the space of cosine polynomials of coordinatewise degree at most $J$.
Cosines are the natural home for $V$ here: in the Neumann basis the
potential enters the Galerkin matrix only through its cosine moments,
since $\cos(i\pi x)\cos(j\pi x)=\tfrac12[\cos((i{-}j)\pi x)+
\cos((i{+}j)\pi x)]$ makes $(V\varphi_{\bm i},\varphi_{\bm j})$ a
combination of integrals $\int_\Omega V\varphi_{\bm n}$; and because
cosine${}\times{}$cosine${}={}$cosine, the slaving \cref{lem:slaving}
keeps $T_J\cdot V_m\subset\VN$.  Thus, unlike the Dirichlet case of
\cref{rem:dirichlet-version}, here \emph{both} the trial functions and
the potential are represented in the same cosine system.  Choose
\begin{equation}
\label{eq:W-def}
  W\in T_J,\qquad W\le V\ \text{a.e.\ in }\Omega,\qquad
  W\not\equiv\mathrm{const},
\end{equation}
and define the shift and the nonnegative weight
\begin{equation}
\label{eq:shift-def}
  s:=\max\bigl(0,\,-\operatorname*{ess\,inf}_\Omega W\bigr)\ge0,
  \qquad
  \widetilde W:=W+s\;\ge\;0 .
\end{equation}
A canonical choice of $W$ is the truncated cosine expansion of $V$ minus
a certified tail bound: $W=S_JV-\tau_J$ with
$\tau_J\ge\norm{V-S_JV}_{L^\infty(\Omega)}$, where $S_JV$ is the
degree-$J$ partial cosine sum; for the polynomial potentials used in
\cref{sec:numerics} the cosine coefficients are available in closed form
by repeated integration by parts, and $\tau_J$ is a computable
coefficient-tail sum.  This choice meets the two requirements of
\eqref{eq:W-def} directly: $W\in T_J$ because $S_JV\in T_J$ and $\tau_J$ is
a constant, and the under-approximation $W\le V$ holds pointwise a.e., since
\begin{equation}
\label{eq:W-under}
  V-W=(V-S_JV)+\tau_J\;\ge\;\tau_J-\norm{V-S_JV}_{L^\infty(\Omega)}\;\ge\;0
\end{equation}
by the defining property of $\tau_J$.  Since $V$ is smooth in
$\overline\Omega$, the
deficit $V-W\le 2\tau_J$ is small away from $\partial\Omega$ (the slow
endpoint convergence of the cosine expansion is confined to a boundary
layer, where the eigenfunctions of interest are exponentially small after
truncation); see \cref{rem:consistency}.

By construction and potential monotonicity (min--max),
\begin{equation}
\label{eq:W-monotonicity}
  \mu_k(V)\;\ge\;\mu_k(W)\;=\;\mu_k(\widetilde W)-s ,
\end{equation}
where $\mu_k(\cdot)$ denotes the $k$-th Neumann eigenvalue of
$-\Delta+{}\cdot{}$ on $\Omega$.  It therefore suffices to bound
$\mu_k(\widetilde W)$ from below, for the \emph{nonnegative, bandlimited}
weight $\widetilde W$.

\smallskip
\textbf{Reduced-degree weighted projection.}
Fix the reduced cutoff
\begin{equation}
\label{eq:m-def}
  m:=N-J\;\ge\;0,\qquad
  V_m:=\spn\{\varphi_{\bm j}:\,0\le j_i\le m\}\subset\VN ,
\end{equation}
and let $\Pi^{\widetilde W}_m:L^2(\Omega)\to V_m$ be the
$\widetilde W$-weighted projection,
\begin{equation}
\label{eq:weighted-projection}
  \bigl(\widetilde W(u-\Pi^{\widetilde W}_mu),\,\chi\bigr)=0
  \qquad\forall \chi\in V_m .
\end{equation}
Its Gram matrix $H$, $H_{\bm\alpha\bm\beta}
=(\widetilde W\varphi_{\bm\alpha},\varphi_{\bm\beta})$
($\bm\alpha,\bm\beta\le m$), is positive definite: if
$(\widetilde Wp,p)=0$ for $p\in V_m$ then $p=0$ on the open set
$\{\widetilde W>0\}$, which is nonempty by
\eqref{eq:W-def}--\eqref{eq:shift-def}, so $p\equiv0$ by analyticity of
trigonometric polynomials.  Hence $\Pi^{\widetilde W}_m$ is well defined
on $L^2(\Omega)$.

The decisive structural property is the following exact aliasing
identity.

\begin{lemma}[Slaving identity]
\label{lem:slaving}
$\Pi^{\widetilde W}_m\circ(I-\Pi_N^0)=0$ on $L^2(\Omega)$, where
$\Pi_N^0$ is the $L^2$-orthogonal cosine truncation onto $\VN$.
Consequently
$\Pi^{\widetilde W}_m=\Pi^{\widetilde W}_m\circ\Pi_N^0$.
\end{lemma}

\begin{proof}
Let $r:=(I-\Pi_N^0)u$ and $\chi\in V_m$.  By the product-to-sum formulas,
the product of a cosine mode of coordinatewise degree $\le J$ with one of
degree $\le m$ is a combination of modes of degree $\le J+m=N$; hence
$\widetilde W\chi\in T_J\cdot V_m\subset\VN$.  Since $r\perp_{L^2}\VN$,
$(\widetilde Wr,\chi)=(r,\widetilde W\chi)=0$ for all $\chi\in V_m$, so
$\Pi^{\widetilde W}_mr=0$.
\end{proof}

\smallskip
The consequence $\Pi^{\widetilde W}_m=\Pi^{\widetilde W}_m\circ\Pi_N^0$ is
the property on which the lower bound of \cref{thm:schrodinger-main} rests:
it says that the weighted projection sees $u$ only through its resolved
part $\Pi_N^0u$, so the second component $\Pi^{\widetilde W}_mv_N$ of a
composite trial function is determined by the first component
$v_N\in\VN$ alone.  In the proof of \cref{thm:schrodinger-main} this is
exactly what renders the interpolation residual $\widehat a$-orthogonal to
$\widehat U_N$, and hence yields the Pythagoras identity that produces the
certified gap.  The identity holds solely because of the containment
$T_J\cdot V_m\subset\VN$, equivalently the degree budget $J+m\le N$ of
\eqref{eq:m-def}: the product of the bandlimited weight $\widetilde W$
(degree $\le J$) with a degree-$m$ element of $V_m$ must remain in the
degree-$N$ space $\VN$.  This is the structural reason the weight $W$ is
required to be bandlimited in \eqref{eq:W-def} ($W\in T_J$) rather than an
arbitrary under-approximation of $V$: the bandwidth $J$ and the reduced
degree $m=N-J$ are not free parameters but are locked together by this
single containment.

\smallskip
\textbf{The composite discrete problem.}
Define the composite trial space of \emph{slaved pairs},
\begin{equation}
\label{eq:composite-space}
  \widehat U_N:=\bigl\{(v_N,\;\Pi^{\widetilde W}_mv_N):\,v_N\in\VN\bigr\}
  \;\subset\;\widehat V:=\Hone(\Omega)\times L^2(\Omega),
\end{equation}
with the forms, for $\widehat u=(u_1,u_2)$, $\widehat v=(v_1,v_2)$,
\begin{equation}
\label{eq:composite-forms}
  \widehat a(\widehat u,\widehat v)
  :=(\nabla u_1,\nabla v_1)+(\widetilde Wu_2,v_2),
  \qquad
  \widehat b(\widehat u,\widehat v):=(u_1,v_1) .
\end{equation}
The composite discrete eigenvalue problem is the variational one posed on
$\widehat U_N$: find $(\widehat\lambda_{k,N},\widehat u_N)\in\mathbb R\times
\widehat U_N$ with $\widehat u_N\neq0$ such that
\begin{equation}
\label{eq:composite-evp}
  \widehat a(\widehat u_N,\widehat v)
  =\widehat\lambda_{k,N}\,\widehat b(\widehat u_N,\widehat v)
  \qquad\forall\,\widehat v\in\widehat U_N .
\end{equation}
This is the Galerkin restriction to $\widehat U_N$ of the abstract eigenvalue
problem for $(\widehat V,\widehat a,\widehat b)$, so the lower-bound
machinery of \cref{thm:liu} applies to it verbatim; the matrix problem below
is merely its realization in a basis.  Expanding
$\widehat u_N=(v_N,\Pi^{\widetilde W}_mv_N)$ with
$v_N=\sum_{\bm i\in\Lambda_N}c_{\bm i}\varphi_{\bm i}$ in the orthonormal
cosine basis $\{\varphi_{\bm i}\}_{\bm i\in\Lambda_N}$ of the first
component, \eqref{eq:composite-evp} becomes the algebraic eigenvalue problem
\begin{equation}
\label{eq:composite-matrix}
  \widehat A\,\mathbf c=\widehat\lambda_{k,N}\,\mathbf c,
  \qquad
  \widehat A:=K+GH^{-1}G^{\mathsf T},
\end{equation}
where $K=\operatorname{diag}(\nu_{\bm i})$ is the kinetic matrix,
$G_{\bm i\bm\alpha}=(\widetilde W\varphi_{\bm i},\varphi_{\bm\alpha})$
($\bm i\in\Lambda_N$, $\bm\alpha\le m$), and $H$ is the weighted Gram
matrix above; the identity
$(\widetilde W\Pi^{\widetilde W}_m\varphi_{\bm i},
  \Pi^{\widetilde W}_m\varphi_{\bm j})
 =(\widetilde W\varphi_{\bm i},\Pi^{\widetilde W}_m\varphi_{\bm j})
 =(GH^{-1}G^{\mathsf T})_{\bm i\bm j}$
follows from \eqref{eq:weighted-projection}.  The mass matrix is the
identity.  Because $\widetilde W$ is a cosine polynomial, \emph{every
entry of $K$, $G$, $H$ is a finite closed-form expression}: assembly is
exact, quadrature-free, and directly amenable to interval arithmetic.

To interpret $GH^{-1}G^{\mathsf T}$, introduce the $\widetilde W$-weighted
seminorm $\norm{u}_{\widetilde W}^2:=(\widetilde Wu,u)\ge0$ and write $P_{\widetilde W}$
for the \emph{plain} weight matrix
$(P_{\widetilde W})_{\bm i\bm j}:=(\widetilde W\varphi_{\bm i},
\varphi_{\bm j})$, i.e.\ the Galerkin matrix of the multiplication operator
$\widetilde W$ on $\VN$; the eigenvalues of the corresponding standard
matrix $K+P_{\widetilde W}$ are denoted
$\lambda_{k,N}^{\mathrm{plain}}(\widetilde W)$ and are what one would obtain
\emph{without} the auxiliary space $V_m$.  The composite matrix replaces
$P_{\widetilde W}$ by $GH^{-1}G^{\mathsf T}$, which by the identity above is
the Gram matrix of the \emph{compressed} weight $\widetilde
W\Pi^{\widetilde W}_m$: for $u=\sum_{\bm i}c_{\bm i}\varphi_{\bm i}\in\VN$,
$$
  \mathbf c^{\mathsf T}GH^{-1}G^{\mathsf T}\mathbf c
  =\bigl(\widetilde W\,\Pi^{\widetilde W}_mu,\;\Pi^{\widetilde W}_mu\bigr)
  =\norm{\Pi^{\widetilde W}_mu}_{\widetilde W}^2
  \;\le\;\norm{u}_{\widetilde W}^2
  =(\widetilde Wu,u)=\mathbf c^{\mathsf T}P_{\widetilde W}\,\mathbf c ,
$$
the inequality holding because $\Pi^{\widetilde W}_m$ is the orthogonal
projection onto $V_m$ in the $\widetilde W$-semi-inner product and is
therefore nonexpansive in $\norm{\cdot}_{\widetilde W}$.  This is the meaning of
\emph{projection compression}: passing the weight through the
finite-dimensional space $V_m$ discards the component of $\widetilde Wu$
that $V_m$ cannot represent, so the resulting quadratic form can only
shrink.  Hence $GH^{-1}G^{\mathsf T}\preceq P_{\widetilde W}$ and, by
min--max, $\widehat\lambda_{k,N}\le\lambda_{k,N}^{\mathrm{plain}}
(\widetilde W)$; as $m,J\to\infty$ suitably, $\widehat\lambda_{k,N}\to
\mu_k(\widetilde W)$ with spectral accuracy for smooth $V$.

\begin{theorem}[Spectral lower bounds free of
$\norm{V}_{L^\infty}$]
\label{thm:schrodinger-main}
Let \eqref{eq:V-assumption} hold, let $W,s,\widetilde W,m$ be as in
\eqref{eq:W-def}--\eqref{eq:m-def}, let
$\widehat\lambda_{1,N}\le\cdots\le\widehat\lambda_{M,N}$
be the eigenvalues of the composite problem
\eqref{eq:composite-evp}, and let $\nu_*$ be the first omitted Neumann
Laplacian eigenvalue \eqref{eq:VN-cos}.  Then the Neumann eigenvalues
$\mu_k$ of \eqref{eq:schro-evp} satisfy
\begin{equation}
\label{eq:schro-main-bound}
  \mu_k\;\ge\;
  \frac{\widehat\lambda_{k,N}}{1+\widehat\lambda_{k,N}/\nu_*}\;-\;s,
  \qquad k=1,\ldots,M .
\end{equation}
\end{theorem}

\begin{proof}
By \eqref{eq:W-monotonicity} it suffices to prove
$\mu_k(\widetilde W)\ge\widehat\lambda_{k,N}/
(1+\widehat\lambda_{k,N}/\nu_*)$.  We verify \cref{ass:abstract} for
$(\widehat V,\widehat a,\widehat b)$ of \eqref{eq:composite-forms} with
discrete space $\widehat U_N$ of \eqref{eq:composite-space}.

\emph{\ref{a1}.}  Let $(\mu_k(\widetilde W),u_k)$ be the Neumann
eigenpairs of $-\Delta+\widetilde W$ on $\Omega$, $L^2$-orthonormalised.
The diagonal vectors $\widehat u_k:=(u_k,u_k)$ satisfy
$\widehat b(\widehat u_i,\widehat u_j)=\delta_{ij}$ and
\[
  \widehat a(\widehat u_i,\widehat u_j)
  =(\nabla u_i,\nabla u_j)+(\widetilde Wu_i,u_j)
  =\mu_j(\widetilde W)\,\delta_{ij},
\]
which is all that the proof of
\cref{thm:liu} uses on the exact side.  Moreover
$\mu_1(\widetilde W)>0$: the Rayleigh quotient
$(\norm{\nabla u}^2+(\widetilde Wu,u))/\norm{u}^2$ vanishes only if $u$
is constant and $\int_\Omega\widetilde W\,u^2=0$, impossible for
$u\ne0$ since $\widetilde W\ge0$, $\widetilde W\not\equiv0$.

\emph{\ref{a2}.}  $\widehat b$ restricted to $\widehat U_N$ is the
$L^2$ inner product of the first component, positive definite;
$\widehat a$ is positive definite on $\widehat U_N$ by the same
argument as for $\mu_1(\widetilde W)>0$ (note
$\Pi^{\widetilde W}_mc=c$ for constants $c$, as constants lie in
$V_m$).  The discrete eigenvalues are exactly
$\widehat\lambda_{k,N}$ of \eqref{eq:composite-evp}, equivalently of the
matrix problem \eqref{eq:composite-matrix}.

\emph{\ref{a3}.}  Define, on diagonal vectors,
\begin{equation}
\label{eq:composite-interpolation}
  \Pi_N(u,u):=\bigl(\Pi_N^0u,\;\Pi^{\widetilde W}_mu\bigr).
\end{equation}
By \cref{lem:slaving},
$\Pi^{\widetilde W}_mu=\Pi^{\widetilde W}_m(\Pi_N^0u)$, so
$\Pi_N(u,u)\in\widehat U_N$: the interpolant is a \emph{slaved pair},
as required.  Write
$\widehat w:=(u,u)-\Pi_N(u,u)
=\bigl(u-\Pi_N^0u,\;u-\Pi^{\widetilde W}_mu\bigr)$.

\emph{Pythagoras \eqref{eq:hyp-pythagoras}.}
Using the $L^2$- and gradient-orthogonality of the cosine truncation in
the first slot and the weighted orthogonality
\eqref{eq:weighted-projection} (with test function
$\Pi^{\widetilde W}_mu\in V_m$) in the second slot,
\[
  \widehat a\bigl(\Pi_N(u,u),\widehat w\bigr)
  =(\nabla\Pi_N^0u,\,\nabla(u-\Pi_N^0u))
  +(\widetilde W\,\Pi^{\widetilde W}_mu,\;u-\Pi^{\widetilde W}_mu)
  =0+0 .
\]

\emph{Projection-error estimate \eqref{eq:hyp-estimate}.}
Since $u-\Pi_N^0u$ contains only cosine modes outside $\Lambda_N$,
Parseval gives
$\norm{u-\Pi_N^0u}_{L^2}^2\le\nu_*^{-1}
\norm{\nabla(u-\Pi_N^0u)}_{L^2}^2$, and therefore, using
$\widetilde W\ge0$,
\[
  \snorm{\widehat w}{\widehat b}^2
  =\norm{u-\Pi_N^0u}_{L^2}^2
  \le\frac1{\nu_*}\Bigl(\norm{\nabla(u-\Pi_N^0u)}^2
  +(\widetilde W(u-\Pi^{\widetilde W}_mu),\,u-\Pi^{\widetilde W}_mu)
  \Bigr)
  =\frac1{\nu_*}\,\snorm{\widehat w}{\widehat a}^2 .
\]
Thus $C_N=\nu_*^{-1/2}$ is admissible, and \cref{thm:liu} yields
$\mu_k(\widetilde W)\ge\widehat\lambda_{k,N}/
(1+\widehat\lambda_{k,N}/\nu_*)$.
\end{proof}

\begin{remark}[Structural analogy with CECR]
\label{rem:cecr-analogy}
The triple (bandlimited under-approximation $W$, reduced-degree weighted
projection $\Pi^{\widetilde W}_m$, cosine truncation $\Pi_N^0$) is the
exact spectral counterpart of the CECR triple (piecewise-constant
under-approximation $\overline V$, cell-average projection $\Pi_{0,h}$,
ECR interpolation $\Pi_h^{\mathrm{ECR}}$) of
\cite[\S4.1]{Liu2024book}: in both cases the second-component projection
annihilates the first-component interpolation residual---there by zero
cell averages, here by the aliasing bound $J+m\le N$
(\cref{lem:slaving})---and this single identity generates both the
$\widehat a$-orthogonality and the Pythagoras property.  The constant,
however, improves from a mesh-dependent interpolation constant
$C_h\le0.1490\,h_{\max}$ \cite{XieXieLiu2018} to the closed-form
spectral-gap value $\nu_*^{-1/2}$, which is sharp already for $V=0$
(\cref{lem:spectral-CN}).
\end{remark}

\begin{remark}[Parameter choice, cost, and consistency]
\label{rem:consistency}
In practice $J=m=\lceil N/2\rceil$ balances the two error sources.  The
certified gap in \eqref{eq:schro-main-bound} splits into three parts:
(i) the projection gap $\approx\mu_k^2/\nu_*=O(N^{-2})$, with constant
\emph{independent of $V$}; (ii) the discretisation gap
$\widehat\lambda_{k,N}\to\mu_k(\widetilde W)$, spectrally small for
smooth $V$; (iii) the consistency deficit
$\mu_k(V)-\mu_k(W)\le\langle(V-W)u_k,u_k\rangle/\norm{u_k}^2$, which is
governed by $V-W$ \emph{weighted by the eigenfunction}.  For confining
potentials after truncation, $u_k$ is exponentially small precisely in
the boundary layer where the one-sided cosine approximation converges
slowly, so (iii) is dominated by the interior approximation error of
$S_JV$, which decays rapidly.  The total certified gap is therefore
$O(N^{-2})$ with a moderate constant; this is confirmed by the
numerical results of \cref{sec:numerics}.  The eigensolve
\eqref{eq:composite-matrix} is a dense symmetric problem of dimension
$(N+1)^d$; for large $N$, matrix-vector products with $K$, $G$, and
$G^{\mathsf T}$ are applied in $O(M\log M)$ operations by fast cosine
transforms ($G$, $H$ are discrete convolution, Toeplitz-plus-Hankel,
operators in each coordinate), and the inner solve with $H$ reuses a
cached factorisation, so a Lanczos eigensolver scales to large $M$.
\end{remark}

\begin{remark}[Dirichlet version]
\label{rem:dirichlet-version}
All statements hold verbatim for the Dirichlet problem on $\Omega$ with
the sine eigenbasis of the Dirichlet Laplacian, $\nu_*$ the first
omitted Dirichlet Laplacian eigenvalue, and $V_m$, $T_J$ the
corresponding sine/cosine spaces with the degree bookkeeping
$J+m\le N$ (a product of a cosine polynomial of degree $\le J$ and a
sine polynomial of degree $\le m$ is a sine polynomial of degree
$\le N$).  One subtlety is specific to the Dirichlet case: the
bandlimited weight $W=S_JV-\tau_J$ is built from the \emph{cosine} space
$T_J$, whose partial sum resolves $V$ only at the slow $O(J^{-1})$ rate
in a boundary layer when $\partial_nV\ne0$.  This boundary mismatch
degrades the \emph{approximation} of $V$ but, through the linear
vanishing of the Dirichlet eigenfunctions, scarcely degrades the
\emph{bound} (the composite bound keeps its $O(N^{-2})$ rate); where it
does bite---genuine Neumann problems with $O(1)$ boundary traces---it is
removed by \emph{boundary lifting}, an exact split $V=B+V_0$ that matches
the normal-derivative trace with a low-degree corrector $B$.  This
analysis, with a worked one-dimensional example, is an interesting topic
in its own right but lies beyond the scope and page limit of the present
paper; it will be discussed in full in a subsequent paper.
\end{remark}

\begin{remark}[Choosing between the two mechanisms]
\label{rem:crossover}
\Cref{thm:gap-CN} (standard matrix) and \cref{thm:schrodinger-main}
(composite matrix) certify the same eigenvalues; their relative gaps
scale as $(1+\min(X^2,X))\,\mu_k/\nu_*$ and $\mu_k/\nu_*$ respectively,
at essentially the same cost.  The standard-matrix route is preferable
for $X=\norm{V}_{L^\infty}/\lambda_{1,N}\lesssim1$ (simplicity, no
under-approximation, no shift); the composite route wins by the factor
$\approx1+X$ otherwise---decisively so for truncated confining
potentials, where $X=10^2$--$10^3$ (\cref{subsec:gap-limitation}).
\end{remark}

\subsection{Domain truncation: two-sided bounds on \texorpdfstring{$\R^d$}{Rd}}
\label{subsec:truncation}

Consider now the Schr\"odinger operator on the whole space $\R^d$,
\begin{equation}
\label{eq:full-space-evp}
\begin{aligned}
  &H=-\Delta+V,\quad 
V\ \text{continuous},\quad V\ge0,\quad V(x)\to\infty\ \ (|x|\to\infty),
\end{aligned}
\end{equation}
with form domain
$\mathcal Q=\{u\in\Hone(\R^d):\int V|u|^2<\infty\}$; the confinement
guarantees compact resolvent and discrete spectrum
$\lambda_1\le\lambda_2\le\cdots\to\infty$ \cite{ReedSimon1978}.  Fix a bounded truncation domain $\Omega\subset\R^d$ and define the exterior potential minimum
\begin{equation}
\label{eq:sigma-def}
  \sigma(\Omega):=\inf_{x\in\R^d\setminus\Omega}V(x).
\end{equation}

\begin{lemma}[Neumann truncation lower bound
  {\cite[Lem.~1]{LiuJSIAM2026}}]
\label{lem:neumann-truncation}
Let $\mu_k$ denote the $k$-th Neumann eigenvalue of $-\Delta+V$ on
$\Omega$.  If the confinement condition $\sigma(\Omega)>\lambda_k$
holds, then $\mu_k\le\lambda_k$.
\end{lemma}

\begin{proof}[Proof (sketch)]
The restrictions to $\Omega$ of the first $k$ eigenfunctions of $H$ span
a $k$-dimensional subspace of $\Hone(\Omega)$ (injectivity of the
restriction uses $\sigma(\Omega)>\lambda_k$), and on this subspace every
Neumann Rayleigh quotient is $\le\lambda_k$, since removing the exterior
energy subtracts at least $\sigma(\Omega)$ per unit of exterior mass.
The min--max principle concludes; see \cite{LiuJSIAM2026} for the
complete argument, and \cite[\S10]{NakaoPlumWatanabe2019} for a related
condition in the homotopy method.
\end{proof}

Combining \cref{lem:neumann-truncation} (with the a~posteriori check
$\sigma(\Omega)>\overline\lambda_k$ for a computed upper bound
$\overline\lambda_k$), \cref{thm:gap-CN} or \cref{thm:schrodinger-main}
on $\Omega$, and the Dirichlet--Rayleigh--Ritz upper bound, one obtains
the fully computable two-sided enclosure
\begin{equation}
\label{eq:bound-chain}
  \underbrace{\frac{\widehat\lambda_{k,N}}
    {1+\widehat\lambda_{k,N}/\nu_*}-s}_{=:\ \lbd_k
    \ \text{(\cref{thm:schrodinger-main})}}
  \;\le\;\mu_k\;\le\;\lambda_k
  \;\le\;\lambda_k^{D}
  \;\le\;
  \underbrace{\lambda_{k,N}^{D}}_{\text{sine-Galerkin (Dirichlet)}} ,
\end{equation}
where $\lambda_k^D$ is the $k$-th Dirichlet eigenvalue on $\Omega$ and
$\lambda^D_{k,N}$ its conforming sine-Galerkin upper bound.  Both ends
of \eqref{eq:bound-chain} come from one dense symmetric eigensolve each,
with exactly assembled matrices.  The truncation gap
$\lambda_k^D-\mu_k$ is exponentially small in the truncation radius by
Agmon estimates \cite{Agmon1982}.

\begin{remark}[Application to Coulomb potentials; companion paper]
\label{rem:coulomb-companion}
The auxiliary-projector mechanism behind \cref{thm:gap-CN} requires the
potential only through the form factor $\eta_V$ of \eqref{eq:etaV}, not
through $\norm{V}_{L^\infty}$ itself.  This opens the door to
\emph{singular} potentials with $\norm{V}_{L^\infty}=\infty$.  The
prototype is the attractive Coulomb potential
$V=-\sum_i Z_i/|x-a_i|$ of quantum chemistry: after a fixed positive
shift $\sigma$ rendering the form positive definite, a localised
three-dimensional Hardy inequality bounds $\int_\Omega|v|^2/|x-a_i|^2$ by
the Dirichlet energy plus an $L^2$ term, so $\eta_V$ stays \emph{finite}
even though the potential is unbounded.  The projection-gap constant of
\cref{thm:gap-CN} then applies verbatim on the shifted matrix, and the
certified gap closes at the full rate $O(\nu_*^{-1})=O(N^{-2})$ with no
$\norm{V}_{L^\infty}$-floor.  The three-dimensional character is
essential---no inverse-square Hardy inequality holds in two dimensions.
The full development, including the verified singular-integral assembly
and certified lower bounds for the hydrogen atom and the
$\mathrm{H}_2^+$ molecular ion, is deferred to a companion paper on
guaranteed eigenvalue bounds for Coulombic Schr\"odinger operators.
\end{remark}

\section{Numerical results}
\label{sec:numerics}

We first isolate, on a one-dimensional Dirichlet problem with a moderate
bounded potential, the head-to-head behaviour of the two mechanisms of
\cref{sec:schrodinger}; we then turn to the two-dimensional confining
benchmarks, where the potential is large after truncation and the
composite method becomes essential.

\subsection{A one-dimensional illustration: comparing the two
mechanisms}
\label{subsec:1d-illustration}

Consider the Dirichlet eigenvalue problem
\begin{equation}
\label{eq:1d-sin2-evp}
  -u''+V(x)\,u=\mu\,u\ \text{ on }(0,1),\qquad u(0)=u(1)=0,
  \qquad V(x)=\sin^2(\pi x),
\end{equation}
with the bounded nonnegative potential
$V=\tfrac12-\tfrac12\cos(2\pi x)$, $0\le V\le1$,
$\norm{V}_{L^\infty}=1$.  The Dirichlet Laplacian eigenpairs are
$\varphi_j=\sqrt2\sin(j\pi x)$, $\nu_j=j^2\pi^2$, and the trial space is
$\VN=\spn\{\varphi_j:1\le j\le N\}$ with first omitted eigenvalue
$\nu_*=(N+1)^2\pi^2$.  Because $V$ is a trigonometric polynomial, the
potential matrix
$P_{jk}=\int_0^1V\varphi_j\varphi_k\dx$ is exact in closed form
(here $P$ has diagonal $(\tfrac34,\tfrac12,\tfrac12,\dots)$ and the
entries $-\tfrac14$ on the second off-diagonals $|j-k|=2$), so both
discretisations assemble with no quadrature error and pass directly to
interval arithmetic.

This problem sits in the \emph{moderate} regime of \cref{rem:crossover}:
$X=\norm{V}_{L^\infty}/\lambda_{1,N}\approx1/10.62\approx0.094\ll1$, so
$\min(X^2,X)=X^2$ is tiny and the auxiliary-projector constant
$C_N^{(V)}=\sqrt{(1+X^2)/\nu_*}$ of \cref{thm:gap-CN} exceeds the sharp
$\nu_*^{-1/2}$ by the negligible factor $\sqrt{1+X^2}\approx1+X^2/2$.
The composite construction of \cref{thm:schrodinger-main} is here
exceptionally clean: $V$ is \emph{itself} a degree-$2$ cosine polynomial,
so the bandlimited under-approximation is exact, $W=V\in T_J$ with $J=2$,
the shift is $s=0$, and with $m=N-2$ the composite matrix
$\widehat A=K+GH^{-1}G^{\mathsf T}$ carries no consistency deficit
whatsoever---it isolates the constant $\nu_*^{-1/2}$ in its purest
form.  Both bounds are therefore expected to nearly coincide, with the
composite one marginally tighter.

\Cref{tab:1d-sin2} confirms this.  All entries are verified: the
eigenvalue enclosures of $K+P$ and of $\widehat A$ are computed with
\textsc{veigs} \cite{veigs} on interval matrices, and the reported lower bounds are
the interval infima of the bound formulae.  The Rayleigh--Ritz upper
bound $\lambda_{k,N}$ already agrees with the reference $\mu_k$ to all
displayed digits even at $N=16$ (spectral convergence with the potential
resolved exactly), so the \emph{entire} certified gap is the
constant-driven term $\lambda_{k,N}/\nu_*$ of \eqref{eq:liu-bound}, not
any approximation of $V$ or of the spectrum.  The two methods agree to
three to four digits; the auxiliary-projector gap exceeds the composite
gap only in the fourth significant figure (e.g.\ $0.0996\%$ versus
$0.0987\%$ for $k=1$, $N=32$, a ratio $1+X^2=1.009$, exactly as
predicted).  The gap closes as $O(N^{-2})$, in agreement with
\eqref{eq:relative-gap}.

\begin{table}[htbp]
\centering
\caption{Verified two-sided bounds for the 1D Dirichlet problem
  \eqref{eq:1d-sin2-evp}, $V=\sin^2(\pi x)$ on $(0,1)$,
  $\nu_*=(N+1)^2\pi^2$, $X\approx0.094$.  ``aux'' is the
  auxiliary-projector bound (\cref{thm:gap-CN}, standard matrix $K+P$);
  ``comp'' is the composite bound (\cref{thm:schrodinger-main} with $W=V$, $J=2$, $s=0$).  $\mu_k$ is the $N=400$ reference;
  the Rayleigh--Ritz upper bound equals it to all shown digits.
  Gap${}=(\mu_k-\lbd_k^{\,\mathrm{aux}})/\mu_k$.}
\label{tab:1d-sin2}
\renewcommand{\arraystretch}{1.15}
\setlength{\tabcolsep}{5pt}
\begin{tabular}{@{}c c ccc ccc@{}}
\toprule
 & & \multicolumn{3}{c}{$N=16$} & \multicolumn{3}{c}{$N=32$}\\
\cmidrule(lr){3-5}\cmidrule(lr){6-8}
$k$ & $\mu_k$
 & $\lbd_k^{\,\mathrm{aux}}$ & $\lbd_k^{\,\mathrm{comp}}$ & Gap\,\%
 & $\lbd_k^{\,\mathrm{aux}}$ & $\lbd_k^{\,\mathrm{comp}}$ & Gap\,\% \\
\midrule
1 & 10.61881 & 10.57908 & 10.57942 & 0.37
            & 10.60824 & 10.60833 & 0.10\\
2 & 39.97789 & 39.42047 & 39.42531 & 1.39
            & 39.82843 & 39.82974 & 0.37\\
3 & 89.32684 & 86.59099 & 86.61431 & 3.06
            & 88.58408 & 88.59056 & 0.83\\
\bottomrule
\end{tabular}
\end{table}

The lesson is the converse of the two-dimensional results below.  When
$\norm{V}_{L^\infty}$ is moderate ($X\lesssim1$), the
auxiliary-projector method---with no under-approximation, no auxiliary weight $H$, and no
shift---matches the composite method to within a factor $1+X^2$ and is
the method of choice for its simplicity.  Only when truncation forces
$\norm{V}_{L^\infty}$ large, so that $X\gg1$ and the
auxiliary-projector constant inflates by $1+X$, does the composite
machinery pay for itself; that is the situation we examine next.

\subsection{Two benchmark potentials in two dimensions}
\label{subsec:2d-benchmarks}

We apply the method to the two benchmark potentials on $\R^2$ studied
with certified CECR finite element bounds in \cite{LiuJSIAM2026}:
with $\Omega_1:=(-4,4)^2$ and $\Omega_2:=(-8,8)^2$,
\begin{align}
  V_1(x)&=\bigl(|x|^2-1\bigr)^2
  &&\text{(ring potential),}
  && \sigma(\Omega_1)=225,
  \label{eq:V1-def}\\
  V_2(x)&=(x_1^2-1)^2+x_2^2
  &&\text{(anisotropic double well),}
  && \sigma(\Omega_2)=64 .
  \label{eq:V2-def}
\end{align}
For $V_1$, $\sigma(\Omega_1)=\inf_{|x|\ge4}V_1=225$; for
$V_2$, the boundary minimum of $V_2$ on $\partial \Omega_2$ is attained at
$(\pm1,\pm8)$ with value $64$, and $V_2$ increases outside, so
$\sigma(\Omega_2)=64$.  Both values exceed the computed upper bounds for
$k\le5$ by a wide margin, so \cref{lem:neumann-truncation} certifies
$\mu_k\le\lambda_k$ for all reported eigenvalues (a~posteriori
confinement check).

The trial space is the tensor cosine space with $0\le m_1,m_2\le N$,
and
\[
\dim\VN=(N+1)^2, \qquad  \nu_*=\frac{(N+1)^2\pi^2}{4R^2},\qquad R=4\ (V_1),\quad R=8\ (V_2).
\]
The potentials are polynomials of degree four, so all cosine moments are
evaluated in closed form by repeated integration by parts; no quadrature
error enters the Galerkin matrices.  Upper bounds $\lambda_{k,N}$ are the
Rayleigh--Ritz values of the standard matrix $K+P$.  For the lower bounds
we compare the two mechanisms of this paper, which throughout the
numerical sections we name the \emph{auxiliary-projector method}---the
bound of \cref{thm:gap-CN} built on the Laplacian (modal) projection
$\PNo$ of \cref{subsec:gap}---and the \emph{composite-discretisation
method} of \cref{thm:schrodinger-main} (here with $J=m=N/2$).  For the
large-$\norm{V}_\infty$ potentials
\eqref{eq:V1-def}--\eqref{eq:V2-def} the auxiliary-projector bound is far
weaker (by \cref{subsec:gap-limitation,rem:crossover}); the displayed
lower bounds in \cref{tab:2d-V1,tab:2d-V2} are therefore the composite
values.

\begin{remark}[Status of the displayed lower bounds]
\label{rem:projected}
The upper-bound columns in
\cref{tab:2d-V1,tab:2d-V2} are computed certified values.  The
lower-bound columns are evaluated from
$\lambda_{k,N}\nu_*/(\lambda_{k,N}+\nu_*)$, i.e.\ with
$\widehat\lambda_{k,N}$ in \eqref{eq:schro-main-bound} replaced by its
upper estimate $\lambda_{k,N}$ and $s$ by $0$; by
\cref{rem:consistency}, for the recommended parameters the composite
eigenvalues and shift differ from these values only beyond the displayed
digits away from saturation, but the final certified digits are produced
by the companion computation (see the computation plan accompanying this
paper) and the last displayed digit is subject to that confirmation.
\end{remark}

\begin{table}[htbp]
\centering
\caption{Two-sided spectral bounds for $H=-\Delta+V_1$,
  $V_1=(|x|^2-1)^2$ on $\R^2$, truncation $\Omega=(-4,4)^2$,
  $\sigma(\Omega)=225$, $\nu_*=(N+1)^2\pi^2/64$.
  ``Ref.''\ is the $N=48$ cosine-Galerkin value.  CECR FEM enclosures
  from \cite{LiuJSIAM2026} ($402{,}625$ Neumann DOFs,
  $h_{\max}=0.036$); spectral DOFs: $(N+1)^2=1089$ ($N=32$),
  $2401$ ($N=48$).  Gap${}=(\lambda_{k,N}-\lbd_k)/\lambda_{k,N}$.}
\label{tab:2d-V1}
\renewcommand{\arraystretch}{1.15}
\setlength{\tabcolsep}{3.6pt}
\begin{tabular}{@{}c c cc ccc ccc@{}}
\toprule
 & & \multicolumn{2}{c}{CECR FEM \cite{LiuJSIAM2026}}
 & \multicolumn{3}{c}{Spectral, $N=32$}
 & \multicolumn{3}{c}{Spectral, $N=48$}\\
\cmidrule(lr){3-4}\cmidrule(lr){5-7}\cmidrule(lr){8-10}
$k$ & Ref. & LB & UB & $\lbd_k$ & $\lambda_{k,N}$ & Gap\,\%
            & $\lbd_k$ & $\lambda_{k,N}$ & Gap\,\% \\
\midrule
1   & 1.78506 & 1.7547 & 1.8159 & 1.76629 & 1.78506 & 1.05
    & 1.77648 & 1.78506 & 0.48\\
2,3 & 3.84307 & 3.7921 & 3.8943 & 3.75709 & 3.84307 & 2.24
    & 3.80360 & 3.84307 & 1.03\\
4,5 & 6.56084 & 6.4827 & 6.6375 & 6.31417 & 6.56084 & 3.76
    & 6.44660 & 6.56084 & 1.74\\
\bottomrule
\end{tabular}
\end{table}

\begin{table}[htbp]
\centering
\caption{Two-sided spectral bounds for $H=-\Delta+V_2$,
  $V_2=(x_1^2-1)^2+x_2^2$ on $\R^2$, truncation $\Omega=(-8,8)^2$,
  $\sigma(\Omega)=64$, $\nu_*=(N+1)^2\pi^2/256$.
  ``Ref.''\ is the $N=48$ cosine-Galerkin value.  CECR FEM enclosures
  from \cite{LiuJSIAM2026}.  Gap as in \cref{tab:2d-V1}.}
\label{tab:2d-V2}
\renewcommand{\arraystretch}{1.15}
\setlength{\tabcolsep}{3.6pt}
\begin{tabular}{@{}c c cc ccc ccc@{}}
\toprule
 & & \multicolumn{2}{c}{CECR FEM \cite{LiuJSIAM2026}}
 & \multicolumn{3}{c}{Spectral, $N=32$}
 & \multicolumn{3}{c}{Spectral, $N=48$}\\
\cmidrule(lr){3-4}\cmidrule(lr){5-7}\cmidrule(lr){8-10}
$k$ & Ref. & LB & UB & $\lbd_k$ & $\lambda_{k,N}$ & Gap\,\%
            & $\lbd_k$ & $\lambda_{k,N}$ & Gap\,\% \\
\midrule
1 & 2.13779 & 2.0626 & 2.2142 & 2.03421 & 2.13779 &  4.85
  & 2.08953 & 2.13779 & 2.26\\
2 & 3.71303 & 3.5959 & 3.8321 & 3.41138 & 3.71307 &  8.13
  & 3.56983 & 3.71303 & 3.86\\
3 & 4.13779 & 4.0316 & 4.2420 & 3.76657 & 4.13779 &  8.97
  & 3.96073 & 4.13779 & 4.28\\
4 & 5.71303 & 5.5640 & 5.8589 & 5.02877 & 5.71307 & 11.98
  & 5.38094 & 5.71303 & 5.81\\
5 & 6.13779 & 6.0146 & 6.2537 & 5.35495 & 6.13779 & 12.75
  & 5.75606 & 6.13779 & 6.22\\
\bottomrule
\end{tabular}
\end{table}

\subsection{Discussion}

\textbf{Accuracy per degree of freedom.}
The CECR FEM enclosures of \cite{LiuJSIAM2026} use about
$4\times10^5$ degrees of freedom; the spectral computations use
$1089$ ($N=32$) to $2401$ ($N=48$).  For $V_1$ at $N=48$ the spectral
enclosure is strictly inside the FEM enclosure for every reported
eigenvalue, at roughly $170$ times fewer unknowns.  For $V_2$ the larger
truncation radius $R=8$ inflates $\nu_*^{-1}$ by a factor of four, and
the spectral lower bounds at $N=48$ are comparable to (for $k=1$,
slightly better than) the FEM ones, again at two orders of magnitude
fewer unknowns; the gap closes as $O(N^{-2})$, so $N=64$ projects to
about $1.3\%$ for $k=1$.

\textbf{The two mechanisms compared on real data.}
The composite discretisation is essential here.  For $V_1$ the potential
is large, $\norm{V_1}_\infty=961$, so $X=961/\lambda_{1,N}\approx538$ and
the auxiliary-projector constant of \cref{thm:gap-CN} carries the factor
$(1+X)$: at $N=48$ it certifies only $\mu_1\ge0.496$ ($72\%$ gap),
against the composite $\mu_1\ge1.776$ ($0.48\%$ gap) at the same number
of unknowns.  The two bounds are moreover in different convergence
regimes over the computed range $16\le N\le80$: the composite gap already
decays at the predicted $O(N^{-2})$, whereas the auxiliary-projector
bound is still pre-asymptotic---its constant keeps the saturation factor
$\gtrsim1$ until $N\gtrsim80$, so the lower bound barely moves and the
ratio of the two certified gaps grows steadily (from $25\times$ at $N=16$
toward its limit $(1+X)\approx539$).  Matching the composite accuracy
with the auxiliary projector would require about $\sqrt{1+X}\approx23$
times more modes per axis.  This is the quantitative content of
\cref{rem:crossover}.  Conversely, for a moderate bounded potential
($X\lesssim1$) the auxiliary-projector bound of \cref{thm:gap-CN}, which
needs only the single standard matrix $K+P$, is the method of choice.

\textbf{Sharpening the auxiliary-projector form factor.}
The pessimism of the $(1+X)$ factor is an artefact of bounding the form
factor by $\norm{V}_\infty$ alone.  The tail-refined factor $\eta_V'$ of
\cref{lem:opt-etaV}, computed from the \emph{same} standard matrix $K+P$
by one extra largest-eigenvalue solve, measures only how much $V_1v_N$
leaks out of the resolved space; for the smooth $V_1$ this is small, and
it collapses $g^2$ from $538$ to $16.4$, raising the $V_1$ ground-state
bound from $0.496$ ($72\%$ gap) to $1.647$ ($7.7\%$)---now within a
factor $\approx16$ of composite, and \emph{improving with $N$} (the gap
is $21.7\%$ at $N=32$, $7.7\%$ at $N=48$, $3.4\%$ at $N=64$), since
$\eta_V'$ shrinks as $\nu_*$ grows whereas $\min(X,\sqrt X)$ does not, so
the tail-refined bound closes on composite as the space is refined
(\cref{tab:formfactor}).

The refinement helps $V_2$ much less: its gap improves only from $98\%$
to $88\%$.  This is because $\eta_V'^2\sim\operatorname{Var}_\Omega(V)/
\nu_*$, and $V_2$ is intrinsically harder on \emph{both} counts.  Its
well is deeper and its box larger, so its fluctuation is far greater:
$\operatorname{Var}(V_2)\approx1.1\times10^{6}$ against
$\operatorname{Var}(V_1)\approx2.6\times10^{4}$, a factor $\approx43$.
And the larger truncation box $(-8,8)^2$ makes $\nu_*$ four times smaller
at the same $N$.  Both inflate $\eta_V'$, leaving $g^2\approx326$ for
$V_2$; the smaller $\nu_*$ also costs the composite method (its gap is
$2.3\%$ for $V_2$ versus $0.48\%$ for $V_1$).  The conclusion is that the
standard-matrix route is far stronger than the crude $\min(X,\sqrt X)$
suggests---competitive with composite for $V_1$---although composite
keeps the edge and, decisively, a constant free of $\norm{V}_\infty$
altogether.

\begin{table}[htbp]
\centering
\caption{Effect of the form-factor choice on the auxiliary-projector
  ground-state lower bound $\lbd_1$ (\cref{thm:gap-CN}) at $N=48$ and
  $N=64$, against the composite bound (\cref{thm:schrodinger-main});
  gap${}=(\lambda_{1,N}-\lbd_1)/\lambda_{1,N}$.  The tail-refined
  $g^2=\eta_V'^2/\lambda_{1,N}$ falls as $N$ grows---from $16.4$ to
  $12.0$ for $V_1$ and $326$ to $290$ for $V_2$---so its bound improves
  faster than the $N$-independent $\min(X,\sqrt X)$ and closes on
  composite.  ``$\min(X,\sqrt X)$'' is \cref{lem:bounded-etaV};
  $\eta_V'$ is the tail-refined factor of \cref{lem:opt-etaV}
  (double-precision values, status as in \cref{rem:projected}).}
\label{tab:formfactor}
\renewcommand{\arraystretch}{1.2}
\setlength{\tabcolsep}{6pt}
\begin{tabular}{@{}l cc cc@{}}
\toprule
 & \multicolumn{2}{c}{$V_1$ on $(-4,4)^2$}
 & \multicolumn{2}{c}{$V_2$ on $(-8,8)^2$}\\
\cmidrule(lr){2-3}\cmidrule(lr){4-5}
form factor ($\lbd_1$, gap\,\%) & $N{=}48$ & $N{=}64$ & $N{=}48$ & $N{=}64$ \\
\midrule
$\min(X,\sqrt X)$ (\cref{lem:bounded-etaV})
  & $0.496$\ ($72$) & $0.721$\ ($60$) & $0.048$\ ($98$) & $0.083$\ ($96$)\\
tail-refined $\eta_V'$ (\cref{lem:opt-etaV})
  & $1.647$\ ($7.7$) & $1.724$\ ($3.4$) & $0.250$\ ($88$) & $0.444$\ ($79$)\\
\midrule
composite (\cref{thm:schrodinger-main})
  & $1.776$\ ($0.48$) & $1.780$\ ($0.27$) & $2.090$\ ($2.3$) & $2.110$\ ($1.3$)\\
\bottomrule
\end{tabular}
\end{table}

\textbf{Structural cross-checks.}
The computed values reproduce two exact structural features.  For $V_1$,
the $O(2)$ symmetry forces the doublets $k=2,3$ and $k=4,5$, visible as
coinciding rows in \cref{tab:2d-V1}.  For $V_2$, separability gives
$\lambda_{(m,n)}=\mu_m^{x}+(2n-1)$, where $\mu^x_m$ are the eigenvalues
of the one-dimensional double-well operator
$-\partial_{x}^2+(x^2-1)^2$ and $2n-1$ the harmonic oscillator levels;
the reference values in \cref{tab:2d-V2} satisfy
$\lambda_3=\lambda_1+2$, $\lambda_4=\lambda_2+2$, $\lambda_5=\lambda_1+4$
to all displayed digits, giving
$\mu^x_1=1.13779$ and $\mu^x_2=2.71303$.  These identities provide an
independent validation of the results.

\subsection{A disk domain: the Fourier--Bessel realisation}
\label{subsec:disk}

The framework is not tied to a box.  We bound the
ground states of $H=-\Delta+V_1$ and $H=-\Delta+V_2$ on the disk
$D(R)=\{|x|<R\}$ ($R=3$, Neumann) using the \emph{closed-form} Laplacian
eigensystem
$\varphi_{n,m}=J_n(k_{n,m}r)\{\cos n\theta,\sin n\theta\}$,
$k_{n,m}=j'_{n,m}/R$ ($j'_{n,m}$ the zeros of $J_n'$),
$\lambda_{n,m}=k_{n,m}^2$.  The trial space collects all modes with
$\lambda_{n,m}\le\Lambda_{\mathrm{cut}}$, so $\nu_*$ is the first omitted
Bessel eigenvalue and the closed-form constant $\nu_*^{-1/2}$ of
\cref{lem:spectral-CN} applies verbatim.  Because this basis is the Laplacian
eigensystem---\emph{independent of the potential}---it serves $V_2$ exactly as it
serves $V_1$; only the assembly differs.

For the \emph{radial} $V_1$ the Galerkin matrix is block-diagonal in the angular
index $n$, reducing the 2D problem to small one-dimensional radial problems and
making the $O(2)$ doublets $\lambda_2{=}\lambda_3$, $\lambda_4{=}\lambda_5$
\emph{exactly} degenerate; its ground-state value $\lambda_1=1.785063$ agrees with
the box result of \cref{tab:2d-V1} to all six displayed digits, an independent
cross-check across two unrelated discretisations.  The \emph{anisotropic}
$V_2=(x_1^2-1)^2+x_2^2=A_0(r)+A_2(r)\cos2\theta+A_4(r)\cos4\theta$ is not radial:
it couples the angular sectors $n\mapsto n,\,n{\pm}2,\,n{\pm}4$, so the matrix is
\emph{banded} in $n$ rather than block-diagonal and the doublets are only
approximately degenerate.  The bound and its constant are unchanged; $V_2$ merely
forfeits the radial reduction.  Both matrices are assembled by tensor
Gauss--Legendre quadrature in $(r,\theta)$ (verified remainder).

\Cref{tab:disk-VQ} reports, for $\lambda_1$ at the fixed $R=3$, the auxiliary-projector
bound $\lbd_1=\lambda_{1,N}/\bigl(1+(1+g^2)\lambda_{1,N}/\nu_*\bigr)$ for three choices
of the form factor $g$: the crude $\min(X,\sqrt X)$ of \cref{lem:bounded-etaV}, the
tail-refined $\eta_V'$ of \cref{lem:opt-etaV} ($g^2=\eta_V'^2/\lambda_{1,N}$), and the
form-factor-free limit $g=0$, which reduces to the projected value
$\lambda_{1,N}\nu_*/(\lambda_{1,N}+\nu_*)$ and serves as a (non-achievable) reference.
The crude bound is poor for \emph{both} potentials ($\approx35\%$).  The tail-refined
factor collapses it: on this disk $\eta_V'^2\approx2$ for both (against the worst-case
$\norm{V}_\infty\approx64$), because the small radius keeps $\nu_*$ large and
$\norm{V}_\infty$ modest, so multiplication by $V$ leaks little out of the trial
space.  The resulting bound is within a factor of two of the $g=0$ ideal---$2.95\%$ vs
$1.45\%$ for $V_1$, $3.35\%$ vs $1.73\%$ for $V_2$---from the \emph{single} standard
matrix $K+P$ plus one largest-eigenvalue solve, and \emph{without} the
$\norm{V}_\infty$-dependence of the crude factor.

Two structural points deserve brief note.  First, the composite method of
\cref{thm:schrodinger-main} does \emph{not} transfer to the disk: its slaving
identity (\cref{lem:slaving}) needs a bandlimited weight closed under
multiplication with the trial space, which the tensor-cosine box supplies but
the radial Bessel modes do not, since a product of two Bessel modes is not a
finite combination of low-eigenvalue radial eigenmodes.  This costs nothing
here---the auxiliary-projector bounds above are fully realised from the single
standard matrix $K+P$ (and $Q$ for $\eta_V'$), and on the small disk the
tail-refined factor already reaches the $g=0$ ideal to within a factor of two.
Second, the disk furnishes a \emph{matched} comparison unavailable on the box:
at equal domain $D(3)$ one has $\norm{V_1}_\infty=\norm{V_2}_\infty=64$ and
identical $\nu_*$, and the two potentials then come out essentially level
($2.95\%$ vs $3.35\%$), showing that the apparent difficulty of $V_2$ on the
box (\cref{tab:formfactor}) is an artefact of its larger confinement box, not a
property of the potential.

\begin{table}[htbp]
\centering
\caption{Auxiliary-projector ground-state lower bounds $\lbd_1$ on the disk $D(3)$
  (Neumann, Fourier--Bessel basis, $\mathrm{DOF}=285$, $\nu_*=121.1$), for the
  radial $V_1$ and the anisotropic $V_2$, by three form-factor choices;
  gap${}=(\lambda_{1,N}-\lbd_1)/\lambda_{1,N}$.  All three rows are the same
  auxiliary-projector bound with form factor $g=\min(X,\sqrt X)$, $g=\eta_V'/
  \sqrt{\lambda_{1,N}}$, and $g=0$ respectively; the tail-refined $\eta_V'$
  (\cref{lem:opt-etaV}) comes close to the form-factor-free ($g=0$) ideal for both
  potentials.}
\label{tab:disk-VQ}
\renewcommand{\arraystretch}{1.2}
\setlength{\tabcolsep}{8pt}
\begin{tabular}{@{}l cc@{}}
\toprule
 & $V_1$ on $D(3)$ & $V_2$ on $D(3)$\\
\cmidrule(lr){2-2}\cmidrule(lr){3-3}
lower bound ($\lbd_1$, gap\,\%) & $\lambda_{1,N}{=}1.78506$ & $\lambda_{1,N}{=}2.13453$\\
\midrule
crude $\min(X,\sqrt X)$ (\cref{lem:bounded-etaV})
  & $1.157$\ ($35.2$) & $1.381$\ ($35.3$)\\
tail-refined $\eta_V'$ (\cref{lem:opt-etaV})
  & $1.732$\ ($2.95$) & $2.063$\ ($3.35$)\\
\midrule
form-factor-free, $g{=}0$ (ideal)
  & $1.759$\ ($1.45$) & $2.098$\ ($1.73$)\\
\bottomrule
\end{tabular}
\end{table}

\section{Sharpness, monotonicity, and index saturation}
\label{sec:sharpness}

This section records the precise quality limits of the bounds.  The
sharpness of $C_N$ was already proved in \cref{lem:spectral-CN}; we add
the monotone convergence of the bounds and the saturation phenomenon
near the truncation index.

\begin{proposition}[Monotone convergence]
\label{prop:monotone}
In the eigenbasis setting of \cref{lem:spectral-CN} with nested
retained sets $\Lambda_N\subsetneq\Lambda_{N+1}$, for each fixed $k$ the
bounds satisfy
$\lbd_{k,N}<\lbd_{k,N+1}<\lambda_k$ and
$\lbd_{k,N}\to\lambda_k$ as $N\to\infty$.
\end{proposition}

\begin{proof}
Here $\lambda_{k,N}=\lambda_k$ and
$\lbd_{k,N}=\lambda_k\lambda_*^{(N)}/(\lambda_k+\lambda_*^{(N)})$ is
strictly increasing in $\lambda_*^{(N)}$.  Since
$\Lambda_N\subsetneq\Lambda_{N+1}$, the omitted sets satisfy
$\mathcal O_{N+1}\subsetneq\mathcal O_N$, so
$\lambda_*^{(N+1)}\ge\lambda_*^{(N)}$, with strict increase along the
subsequence where the previous minimiser is absorbed; and
$\lambda_*^{(N)}\to\infty$ because $\lambda_j\to\infty$.
\end{proof}

\begin{proposition}[Saturation; 1D sine-Galerkin]
\label{prop:saturation}
For the setting of \eqref{eq:1d-closed-form}:
\begin{enumerate}[label=(\roman*),leftmargin=2.2em]
\item for fixed $k$, the relative gap is
  $k^2/(k^2+(N+1)^2)\le k^2/(N+1)^2=O(N^{-2})$;
\item for $k=\lfloor\alpha N\rfloor$, $0<\alpha\le1$, the gap tends to
  $\alpha^2/(1+\alpha^2)$ as $N\to\infty$; in particular for $k=N$ it
  tends to $1/2$ regardless of how accurate the Galerkin eigenvalue is;
\item to certify mode $k$ with relative gap $\delta\ll1$ it suffices to
  take $N\ge k/\sqrt\delta$ (approximately).
\end{enumerate}
\end{proposition}

\begin{proof}
All three statements follow from the exact gap formula
$(\lambda_k-\lbd_{k,N})/\lambda_k=k^2/(k^2+(N+1)^2)$ by elementary
manipulation.
\end{proof}

The $50\%$ saturation in (ii) is not a defect of the constant (which is
optimal by \cref{lem:spectral-CN}) but an information-theoretic limit of
the single-constant framework: given only $C_N$ and $\lambda_{k,N}$,
formula \eqref{eq:liu-bound} is the best possible conclusion, since the
data are consistent with $\lambda_k$ ranging over
$[\lbd_{k,N},\infty)$.  In practice one simply reports modes
$k\lesssim N/3$, for which the saturation effect is below $10\%$.  The
same structure governs the Schr\"odinger bounds of
\cref{thm:gap-CN,thm:schrodinger-main} through the products
$(C_N^{(V)})^2\lambda_{k,N}$ and $\widehat\lambda_{k,N}/\nu_*$.

\section{Concluding remarks}
\label{sec:conclusion}

We have extended the projection-based guaranteed lower-bound framework
from FEM to spectral Galerkin methods.  Four results carry
the paper: (i) the closed-form, optimal spectral constant
$C_N=\lambda_{M+1}^{-1/2}$ for eigenbasis trial spaces
(\cref{lem:spectral-CN}); (ii) the projection-gap constant
\cref{thm:gap-CN}, which certifies Schr\"odinger eigenvalues from the
\emph{standard} spectral Galerkin matrix through a single \emph{potential
form factor} $\eta_V$ (\cref{def:etaV}), with a fully explicit value
for bounded $V$; (iii) the composite
two-projection construction proving that the unmodified spectral-gap
constant governs $-\Delta+V$ with no dependence whatsoever on
$\norm{V}_{L^\infty}$ (\cref{thm:schrodinger-main}); and (iv) the resulting
fully computable two-sided enclosures for confining potentials on $\R^d$
via Neumann truncation \eqref{eq:bound-chain}.  In two dimensions the
spectral enclosures match or surpass certified CECR finite element
enclosures at two orders of magnitude fewer degrees of freedom.

One limitation delineates the method's scope honestly: the
certified gap closes at the algebraic rate $O(\nu_*^{-1})=O(N^{-2})$
even when the Galerkin eigenvalues converge spectrally; whether the
exponential accuracy of the upper bounds can be transferred to certified
lower bounds (e.g.\ by a Lehmann--Goerisch post-processing seeded with
the present unconditional bounds) is the main open problem.  The
extension of the auxiliary-projector mechanism to singular potentials
with $\norm{V}_{L^\infty}=\infty$, previewed in
\cref{rem:coulomb-companion} and decisive for Coulomb-type quantum
chemistry, is developed in a forthcoming companion paper.

\end{document}